\documentclass[3p]{elsarticle}

\usepackage{lineno,hyperref}
\modulolinenumbers[5]

\usepackage{tikz}
	\usetikzlibrary{shapes,arrows}

\usepackage{amsmath, amssymb,amsthm}
\usepackage{dsfont}
\usepackage{caption}
\usepackage{algorithm}
\usepackage{algorithmic}

\newtheorem{theorem}{Theorem}

\newcommand{\elrmat}[3]{#3 \in \mathbb{R}^{#1 \times #2}}

\newcommand{\idn}{\mathds{1}}

\newcommand{\pmat}[1]{\begin{pmatrix} #1 \end{pmatrix}}

\newcommand{\integ}[4]{\int_{#1}^{#2}#3 \text{d}#4}

\newcommand{\hil}[1]{\mathcal H_{#1}}

\newcommand{\nrm}[1]{\lVert {#1} \rVert}

\newcommand{\fig}{Figure~}

\def \foldpicpap {./}

\DeclareMathOperator*{\argmax}{argmax}

\begin{document}
\begin{frontmatter}

\title{Model order reduction of hyperbolic systems at the example of district heating networks}

\author[m1,m2]{Markus Rein\corref{me}}
\cortext[me]{Corresponding author}
\ead{markus.rein@itwm.fhg.de}

\author[m2]{Jan Mohring}
\author[m1]{Tobias Damm}
\author[m1]{Axel Klar}

\address[m1]{TU Kaiserslautern, Erwin-Schr{\"o}dinger-Stra{\ss}e 1,  67663 Kaiserslautern, Germany}
\address[m2]{Fraunhofer Institute for Industrial Mathematics ITWM, Fraunhofer-Platz 1, 67663 Kaiserslautern, Germany}

\begin{abstract}
In this article a framework for the generation of a computationally fast surrogate model for district heating networks is presented. An appropriate model results in an index-1 hyperbolic, differential algebraic equation quadratic in state, exhibiting several hundred of outputs to be approximated. We show the existence of a global energy matrix which fulfills the Lyapunov inequality ensuring stability of the reduced model. By considering algebraic variables as parameters to the dynamical transport, the reduction of a linear, time varying (LTV) problem results. We present a scheme to efficiently combine linear reductions to a global surrogate model using a greedy strategy in the frequency domain. The numerical effectiveness of the scheme is demonstrated at different, existing, large scale networks.\\
\end{abstract}


\end{frontmatter}

%\linenumbers

\section{Introduction}\label{sec:int}
A central challenge of today's world is the efficient supply and use of energy. With the increasing share of renewable energies, energy supply and the dynamics of its transportation networks such as water, electricity and heating networks, became more volatile and diverse. Based on this development, guaranteeing a robust control and efficient use of energy requires precise modeling of the transportation processes. Model predictive control is a central part in efficiently seeking the potential of renewable energies. The corresponding mathematical task is challenging, since it involves large-scale dynamical networks subject to many possible scenarios \cite{hovland_explicit_2008}. These require multiple simulations of the dynamics over large time horizons. Hence the formulation of a numerically efficient and stable surrogate model is an important building block \cite{Benner2014PDE}.\\

Here port-Hamiltonian systems proved to be useful being formulated close to the underlying physical conservation laws \cite{van_der_schaft_port-hamiltonian_2013} while encorporating desired properties such as stability and passivity. Moreover it can be shown that these properties are passed to a reduced model obtained by Galerkin projections if the Hamiltonian energy matrix is included in the reduction process \cite{gugercin_interpolation-based_2009}.\\

We focus on the model order reduction of systems of hyperbolic differential equations \cite{leveque_numerical_2008} at the example of district heating networks. Such networks perform the transport of thermal energy from a centralized power plant to consumers using a network of pipelines. For each of the connected houses, a heat exchanger covers the time dependent power demand of customers by regulating the volume flow based on the currently available thermal energy. Due to its high flexibility towards the injection of different forms of energy, district heating has gained increasing importance for the supply with renewable energies \cite{rezaie_district_2012}. The reduction of hyperbolic systems is a challenging task, since the singular value decay of modes obtained from time-domain snapshots of their dynamics is known to be small. Furthermore the dynamics is non-linear and the state space is high dimensional \cite{sandou_predictive_2005}. In the following we present ideas how to tackle these difficulties in the generation of an efficient surrogate model.\\

The contribution is structured as follows. After presenting the differential algebraic equation (DAE) in section \ref{sec:model}, we formulate the problem setting in section \ref{sec_problem}. To allow for the application of linear reduction methods, section \ref{sec_param} depicts a parametric representation, where the dynamical energy transport is a linear time varying system, altered by the time varying flow variables, modeled as a parameter vector. For the upwind scheme, the construction of a global Lyapunov function with Kernel $Q$ is provided in section \ref{sec:stab} including a proof that it fulfills the corresponding Lyapunov inequality. Subsequently, we show how to systematically construct a global Galerkin projection for a given network topology in section \ref{sec_buildROM} and its effectiveness is demonstrated for two real world networks of large scale in section \ref{sec_numerics}.

\section{Model for district heating}\label{sec:model}
The transport of the energy density (total energy per unit volume) $\varphi$ within a pipeline is modeled by one-dimensional Euler equations. Since water in the liquid phase is the transport medium, the incompressible limit is assumed, simplifying the conservation of mass to $v_x = 0$. The remaining Euler equations for conservation of momentum and inner energy read

\begin{eqnarray}
0 &=& p_x + \frac{\lambda \rho}{2d}|v|v + \rho g h_x\label{eq_consmom}\\
\dot \varphi &=& -v \varphi_x - \frac{4k}{d}(T(\varphi)-T_e)\label{eq_consen}.
\end{eqnarray}

The change of pressure $p_x$ over a pipeline is modeled by frictional forces according to the Darcy-Weisbach equation, where $\lambda$ is a dimensionless friction factor, $d$ is the pipeline diameter, $v$ the advection velocity, and $\rho$ the density. The quantities $\rho$, and $\lambda$ are assumed to be constants within this contribution. Gravitational forces are captured by the height difference $h_x$, and the gravitational constant $g$. For the typical dynamics of heating networks, acceleration is small compared to friction and gravitation, which is why it is neglected in (\ref{eq_consmom}), $\dot v = 0$. This makes (\ref{eq_consmom}) an algebraic equation after integration over the pipeline length. The advection of the energy density $\varphi$ in (\ref{eq_consen}) incorporates an additional sink term due to conduction of heat with transfer coefficient $k$ to the environment with temperature $T_e$.\\

To allow for a numerical treatment of the partial differential equation (PDE), we perform a spatial discretization of (\ref{eq_consen}) employing the upwind scheme yielding a total number of $n$ finite volume cells. With abuse of notation, we denote both the PDE variable and the vector of finite volume cells by $\varphi$. For the description of the network, we introduce the set $\mathcal{E}$ containing $E$ edges which represent all pipelines. More specifically, pipeline $i\in \mathcal{E}$ contains the local set of cells $\mathcal{N}_i=[1,..,n_i]$ with cardinal number $n_i$. In the following the resulting system of ordinary differential equations is considered. To connect incoming and outgoing pipelines within the network, additional algebraic constraints at the junctions have to be posed. A prominent choice is the conservation of energy over node $N$ yielding

\begin{eqnarray}\label{eq_encons}
\sum_{i \in N^-} q_i \varphi_{i,1} = \sum_{j \in N^+} q_j \varphi_{j,n_j},
\end{eqnarray}

where $q_i = \Phi_i v_i$ is the volume flow on pipeline $i$ formed by its cross section $\Phi_i$ and velocity $v_i$. $N^-$ and $N^+$ denote edges exiting and entering node $N$.  By instantaneous mixing of energy flows within $N$, the energy density is identical for all outgoing pipes, $\varphi_{i,1} = \varphi^N, \forall i \in N_+$. Here $\varphi_{e,c}$ is the finite volume cell $c$ of edge $e$ in flow direction. Hence, the energy density $\varphi^N$ adjacent to pipe $i$ is given by

\begin{eqnarray}\label{eq_phi_out}
\varphi^N= \frac{\sum_{j \in N_+} \Phi_j v_j \varphi_{j,n_j}}{\sum_{i \in N_-} \Phi_i v_i}.
\end{eqnarray}

Similarly, volume conservation over node $N$ is assumed, yielding

\begin{eqnarray}\label{eq_cons_vol}
\sum_{j \in N_+} \Phi_j v_j = \sum_{k \in N_-} \Phi_k v_k.
\end{eqnarray}

This allows to write the network dynamics as  

\begin{eqnarray}
\dot \varphi &=& A(v) \varphi + B(v) u_T(t),  \label{eq_sysadvec}\\
y &=& C \varphi \label{eq_sysout}, \\
0 &=& v - K q\label{eq_K1},\\
0 &=& G [v_i \cdot |v_i|]_{i \in \mathcal{E}}\label{eq_K2},\\
0 &=& [q_i \cdot y_i]_{i \in \mathcal{H}} - u_H(t), \label{eq_hous}
\end{eqnarray}

where (\ref{eq_sysadvec}) mirrors the advection of the energy density subject to the input energy density $u_T(t)$, with $\elrmat{n}{1}{B(v)}$. The energy densities relevant for the algebraic equations are measured by the operator $\elrmat{o}{n}{C}$. Both the upwind scheme and the conservation of energy are encoded in the velocity dependent matrix $\elrmat{n}{n}{A(v)}$. Eq. (\ref{eq_K1}) uses the solution $\elrmat{E}{L}{K}$ of (\ref{eq_cons_vol}), to describe the pipeline velocities by $L$ independent volume flows $q$. Kirchhoff's circuit law presented in (\ref{eq_K2}) claims that the sum of pressure differences over a loop within the network equals 0, where $\elrmat{(L-H)}{E^2}{G}$. Finally in (\ref{eq_hous}), the demanded power consumption $u_H$ is provided by energy density and volume flow at houses $\mathcal{H}$ with cardinal number $H$. Consequently the velocity changes dynamically with the energy density at the houses and their time dependent consumption $u_H(t)$. For the rest of this contribution, we abbreviate the algebraic equations (\ref{eq_K1}) - (\ref{eq_hous}), by the nonlinear system of equations,

\begin{align}\label{eq_algeb}
0 = g(q,y,u_H).
\end{align}

\section{Problem formulation and reduction approaches}\label{sec_problem}
We seek to construct a global, stable, surrogate model of the hyperbolic, quadratic in state index-1 DAE (\ref{eq_sysadvec} - \ref{eq_hous}), which after offline generation, can be reused for all admissible inputs $\mathcal{U_T}$. Consequently, the time to generate the surrogate model is negligible. Note that the surrogate model is generated for an ODE system with a given number of finite volume cells, and not for the PDE system itself. The information necessary to construct the surrogate model is the admissible control $\mathcal{U_T}$, the expected consumption $u_H$, the topology of the network, and the approximation error $\Delta^t$,

\begin{align}\label{eq_err_time}
\Delta^t = \max_{h \in \mathcal{H}}\frac{\nrm{y_h(\cdot)- y_h^r(\cdot)}_2}{\nrm{y_h(\cdot)}_2},
\end{align}

of the simulation in the time domain $I$. Using the maximum relative $l_2$ error ensures that each output is approximated precisely, regardless of its absolute value and other outputs. The reference solution $y$ is obtained by an estimate of the PDE solution as described in section \ref{sec_numerics}. The requirements to the admissible input signals arise from technical restrictions and are summarized in the following space

\begin{align}
\mathcal{U_T} =\{u(t) = c_0  + \sum_{i=1}^{m} c_i \cos(i \omega t + \beta_i), m \omega \leq \hat \omega, \; |\dot u(t)| \leq u_d, \; u(t) \in [u_l,u_h] \; \forall t \in I \}.
\end{align}

The control $u \in \mathcal{U_T}$ is periodic with maximal contributing frequency $\hat \omega$, its absolute temporal derivative is bounded by $u_d$, and its image varies within the interval $[u_l,u_h]$. Note that for the temporal variation of the input energy also the initial state is important which we denote by

\begin{align}
\varphi_0 = \varphi(x,t=0).
\end{align}

Discussing possible reduction approaches \cite{antoul_approx}, balanced truncation proved to be beneficial for linear systems, and is currently even capable of balancing large scale systems in the order of $10^{6}$ states. Since the given problem highly depends on the volume flow field determining the transport on every pipeline, different linear working points will influence the dynamical simulation. As a consequence, focusing on a single linear model is not promising. Dynamic mode decomposition (DMD)\cite{kutz_dynamic_2014}, and its extension to control systems (DMDc)\cite{proctor_dynamic_2016}, aim at approximating a (non-)linear dynamical system by a linear system. In contrast to projection based reductions, the resulting linear operator $A_r$ is constructed by the identified modes and thus is constant in time for the later simulation.
Other techniques calculate a dictionary of solutions in the offline phase and weight them to obtain a minimal $L_1$ error \cite{abgrall_robust_2016}. Reduced basis methods \cite{haasdonk_reduced_2008} aim at approximating the input-to-state map making use of the Proper orthogonal decomposition (POD) technique. The latter bases on snapshots of the system state for time-domain simulations to generate a low dimensional subspace in which the relevant dynamics evolve. For advection dominated phenomena, it is known that the singular value decay of the corresponding snapshots is slow, resulting in large surrogate models. Different approaches were proposed to overcome this problem by shifting the transport dynamics to a common reference frame using the velocity information of the flow field \cite{rim_transport_2018,reiss_shifted_2018,ohlberger_nonlinear_2013}. For the context of energy networks, it suffices to capture the relevant outputs at distinct points in the network. As a consequence, we aim at reproducing the input-to-output map. The reduction of quadratic bilinear differential algebraic equations (QBDAE) system has been addressed in different works \cite{ahmad_moment-matching_2017,benner2018mathcalh_2}. In contrast to these approaches, we want to preserve the algebraic state space variables to guarantee Lyapunov stability in the reduced models. Trajectory piecewise-linear models (TPWL) \cite{rewienski_model_2006} linearize the system dynamics at relevant points in the state space, and reduce the resulting linear systems by projection based methods. In the dynamical simulation, the local reductions are combined according to their distance to the current system state, using an appropriate weighting function. The linearization of the velocity field appearing in the algebraic equations can lead to numerically unstable systems. As a consequence, we keep the product of volume flows $q$ and energy densities and consider $q$ as a parameter set to the dynamical transport. The resulting parameterized, linear time varying system (LTV) is reduced using interpolation based tools from linear model order reduction.

%\cite{tu_improved_2012} Balanced POD

\section{Parameterized description}\label{sec_param}
Using a first order finite volume discretization such as an upwind scheme, the resulting number of differential, energy related state variables $\varphi$ is large compared to the algebraic flow variables $q$. Furthermore, different reduction approaches on differential algebraic systems proved that it is advisable to preserve the algebraic equations when performing model reduction \cite{egger_structure-preserving_2018}. Section \ref{sec:stab} confirms that also in our case, not reducing the algebraic equations is sufficient for stability of the discretized model. As a  consequence, we split the thermal transport from the hydraulic problem and focus on the reduction of the dynamical part. To this end, throughout the rest of the contribution we consider the thermal transport as a parameterized problem,

\begin{equation}\label{eq_syspar}
\begin{aligned}
\dot \varphi &= A(q) \varphi + B_T(q) u_T(t),\\
y &= C \varphi,\\
0 &= g(q,y,u_H).
\end{aligned}
\end{equation}

To reduce the parameter space, the volume flow field $q$ is used as a parameter instead of the velocity field $v$ in (\ref{eq_sysadvec}). The algebraic equations $g$ form a generator system for the parameter trajectory $q(t)$ of the reduced system. The parametric representation of $A,B$ resulting for  the differential system (\ref{eq_sysadvec}) takes the form

\begin{align}\label{param_dep_A}
A(q) &= \sum_{i=1}^{n_f} f_i(q) A_i,\\
B(q) &= \sum_{i=1}^{n_f} f_i(q) B_i,
\end{align}

with $f: \mathbb{R}^{n_q} \rightarrow \mathbb{R}^{n_f}$. This affine representation of the parameter influence allows to compute the reduced system matrices $A_i^r, B_i^r$ in the offline phase and apply the $q$ dependent weights online. Typical heating networks solely contain junctions with two or three coupling pipelines. In this case $f_i(q)=q_i\; \forall i \in \mathcal{H,L}$. Thus the only nonlinear function in the volume flows is introduced by changes in flux direction, and it holds $n_f \approx n_q$.\\

Similar to rational interpolation approaches for the reduction of parameterized linear systems, we consider the parameterized transfer function

\begin{align}\label{eq_localtf}
H(s,q) = C (s \idn - A(q))^{-1} B(q),
\end{align}

as the central object to construct reduced models. We note that (\ref{eq_localtf}) is only well defined for $q$ fixed in time, which is not the case in our dynamical simulation.

\section{Lyapunov stability of the discretized model}\label{sec:stab}
Stability is a key property desired for the ROM. It can be obtained from the FOM if the corresponding energy matrix $Q$ is used in a state space transformation before applying a Galerkin projection \cite{gugercin_interpolation-based_2009}. Hence, it suffices to construct $Q$ for the full order model to conclude for Lyapunov stability of full and reduced model. This is discussed for the upwind scheme in the following section by analyzing its structure. The corresponding discretization of the PDE (\ref{eq_consen}) results in two possible types of cell coupling on a network of pipelines and junctions, compare \fig\ref{fig_illnetwork}. These are coupling with neighboring cells in the pipeline or with border cells coupling to incoming pipelines at junctions,
 
\begin{equation}\label{eq_coupl_typ}
\begin{aligned}
\dot \varphi_{i,j} &= -\frac{v_i}{h_i}(\varphi_{i,j}-\varphi_{i,j-1}),\quad j \in \mathcal{N}_i\\
\dot \varphi_{i,1} &= -\frac{v_i}{h_i}(\varphi_{i,1}-\varphi^{N}(\{\varphi_{j,n_j}|j \in E_i^+\})).
\end{aligned}
\end{equation}

Here $\varphi_{i,j}$ is the finite volume cell $j$ in flow direction of edge $i$. $E_i^+$($E_i^-$) denote the set of edges which enter(exit) the node which edge $i$ couples to.  For a fixed velocity field $v=\bar v$, $A$ is considered as the Jacobian of the ODE system (\ref{eq_sysadvec}),

\begin{eqnarray}
A_{f(i,j),f(k,l)}(\bar v) \equiv \frac{\partial \dot \varphi_{i,j}}{\partial \varphi_{k,l}}(\bar v).
\end{eqnarray}

Rows and columns of the matrix $A$ are mapped to the edge- and cell indices $i,j$ by the ordering function

\begin{eqnarray}
f(e,c) \equiv c + \sum_{k=1}^{e-1} n_k, \quad e \in \mathcal{E}, \quad  c \in \mathcal{N}_e.
\end{eqnarray}

These definitions allow to state the following theorem.\\

\begin{minipage}{.8\textwidth}
\centering
\begin{tikzpicture}[scale=1,transform shape]
\draw[fill] (0,0) circle (0.1cm);
\draw[-](0,0)--node[below]{$\varphi_{i,1}$} (1,0);
\draw[fill] (1,0) circle (0.05cm);
\draw[-](1,0)--node[below]{$\varphi_{i,2}$} (2,0);
\draw[fill] (2,0) circle (0.05cm);
\draw[-] (2,0)--node[below]{$\varphi_{i,3}$} (3,0);
\draw[fill] (3,0) circle (0.05cm);
\draw[dashed](3,0)--(5,0);
\draw[fill] (5,0) circle (0.05cm);
\draw[-](5,0)--node[below]{$\varphi_{i,n_i}$} (6,0);
\draw[fill] (6,0) circle (0.1cm);
\draw[-](6,0)--node[above, rotate=45]{$\varphi_{k_1,1}$} (6.7,0.7);
\draw[fill] (6.7,0.7) circle (0.05cm);
\draw[dashed]((6.7,0.7)--(7.5,1.5);
\draw[-](6,0)-- (6.7,-0.7);
\draw[fill] (6.7,-0.7) circle (0.05cm);
\draw[dashed]((6.7,-0.7)--(7.5,-1.5);
\draw[-](-.7,.7)--node[above, rotate=-45]{$\varphi_{j_1,n_j}$} (0,0);
\draw[fill] (-.7,.7) circle (0.05cm);
\draw[dashed]((-1.5,1.5)--(-.7,.7);
\draw[-](-.7,-.7)--(0,0);
\draw[fill] (-.7,-.7) circle (0.05cm);
\draw[dashed]((-1.5,-1.5)--(-.7,-.7);
\draw(-.5,.1) node{$\vdots$};
\draw(-1,0) node {$E_i^+$};
\draw(6.5,0.1) node {$\vdots$};
\draw(7,0) node {$E_i^-$};
\end{tikzpicture}

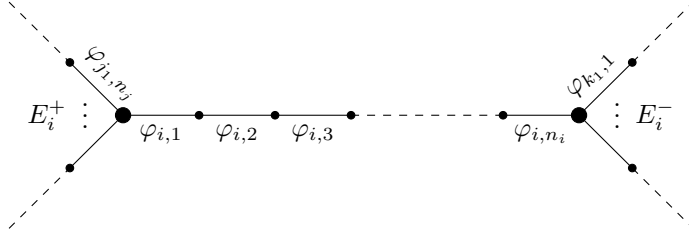
\captionof{figure}{Illustration of the spatial discretization of the PDE using the upwind scheme.\\}
\label{fig_illnetwork}
\end{minipage}

\begin{theorem}\label{theo:stabil}
There exists a global, diagonal, positive definite energy matrix $\elrmat{n}{n}{Q}$, such that for every fixed  set of volume flows $q$ satisfying volume conservation (\ref{eq_cons_vol}), 

\begin{eqnarray}\label{eq:M}
M = (QA(q))^T + (QA(q)) \leq 0.
\end{eqnarray}

Remark 1: $Q$ can be constructed with positive diagonal elements $Q_i \equiv \textit{diag}(Q) =  Q_{f(i,j),f(i,j)} = \Phi_i h_i, \; i \in \mathcal{E}, \; j \in \mathcal{N}_i$. The latter carry the volume $\Phi_i h_i$ of each of the discretization cells on edge $i$.

Remark 2: Since (\ref{eq:M}) holds for all volume flows $q$ satisfying volume conservation (\ref{eq_cons_vol}), $V(\varphi) = \varphi^T Q \varphi$ forms a Lyapunov function for the time-continuous process \cite{antoul_approx}.

Remark 3: Note that also a change in the flux direction which yields a structural modification in the system matrix $A(q)$ leads to a stable system by transformation with $Q$.

\end{theorem}

\begin{proof}\label{proof:lyap}
For the following considerations, we choose the orientation of each edge velocity such that $v_i \geq 0 \;\forall i \in \mathcal{E}$. Hence, velocities and the corresponding volume flows are non-negative. Proving that the symmetric matrix $M$ is negative semi-definite amounts to show that $M$ has non-positive diagonal elements, and is weak (row- and column) diagonally dominant,

\begin{eqnarray}
\sum_{j=1,\; i\neq j}^{n} |M_{ij}| \leq |M_{jj}|.
\end{eqnarray}

With these properties, \cite{horn_matrix_2017} Theorem 6.1.1 allows to conclude that all eigenvalues of $M$ are non-positive and by \cite{horn_matrix_2017} Theorem 4.1.8 a hermitian matrix with non-positive eigenvalues is negative semi-definite.\\

Based on the coupling types (\ref{eq_coupl_typ}), we obtain for $A' \equiv Q A(q)$, $\forall i \in \mathcal{E}$

\begin{eqnarray}
A'_{f(i,j),f(i,j)} &=& -Q_i\frac{v_i}{h_i} = -q_i \label{eq_Adiag}  \\
A'_{f(i,j),f(i,j-1)} &=& Q_i\frac{v_i}{h_i} = q_i, \quad \forall j \in [2,...,n_i] \\
A'_{f(i,1),f(j,n_j)} &=& a'_{ij}=  Q_i\frac{v_i}{h_i}\frac{\Phi_j v_j}{\sum_{k \in E^-_i} \Phi_k v_k} = q_i \frac{q_j}{\sum_{k \in E^-_i} q_k} , \quad \forall j \in E^+_i. \label{eq_junc_connec}
\end{eqnarray}

The quantity $\varphi^N$ leaving node $N$ is replaced by energy conservation (\ref{eq_cons_vol}) to derive the coupling type in (\ref{eq_junc_connec}). The structure of $A'$ is visualized in Figure~\ref{fig_amat}. From (\ref{eq_Adiag}) one easily concludes that the diagonal elements of $M = A' + A'^T$ are non-positive, $M_{f(i,j),f(i,j)} = -2q_i \leq 0, \; \forall i \in \mathcal{E}, \; \forall j \in \mathcal{N}_j$.

\begin{figure}
\centering
\includegraphics[scale=0.75]{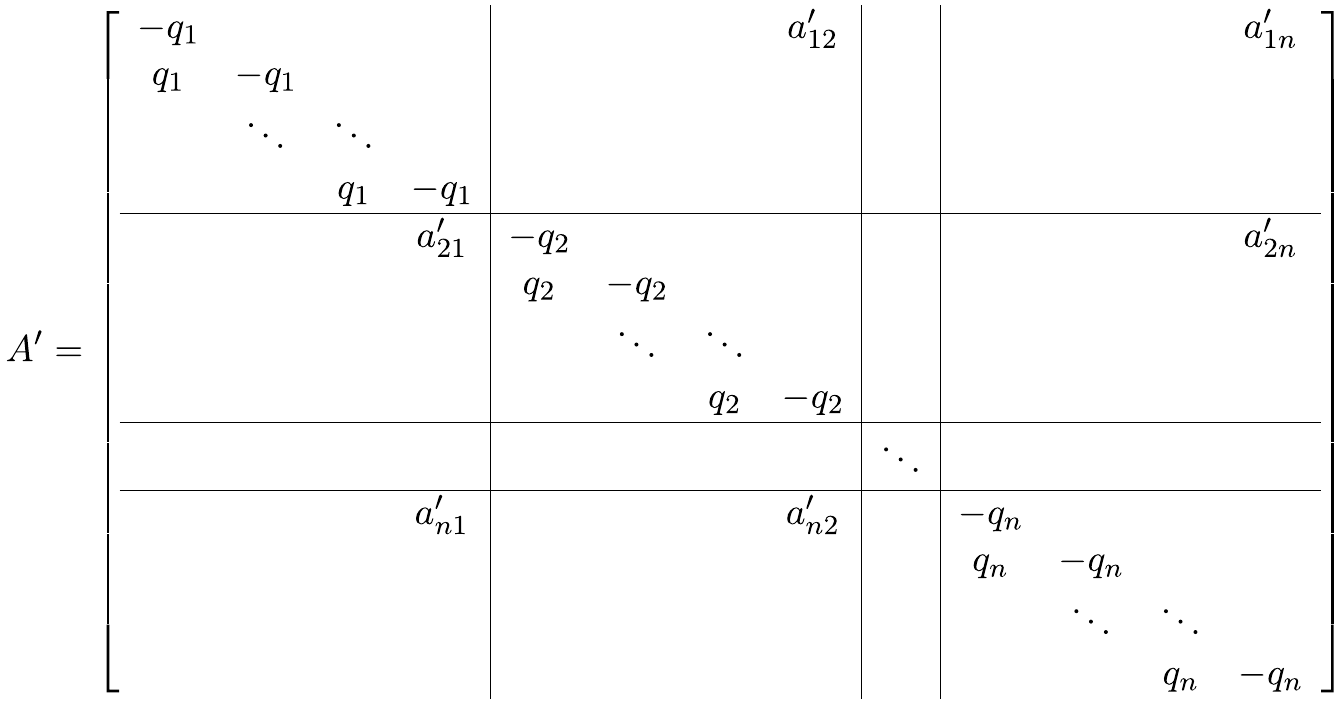}
\caption{Structure of the modified system matrix $A'$ resulting from discretization with the upwind scheme. Off-diagonal elements $a'$ result from incoming and outgoing flows coupling at nodes of the network.}
\label{fig_amat}
\end{figure}

To show that $M$ is weak diagonally row dominant, we start by considering rows which describe inner energy densities in a pipeline. This trivially leads to weak diagonal dominance for these rows, 

\begin{eqnarray}
\sum_{k\neq f(i,j)} |M_{f(i,j),k}| &=& \sum_{k\neq f(i,j)} |A'_{f(i,j),k} + A'_{k,f(i,j)}|\\
&=&  2q_i.
\end{eqnarray}

Focusing on rows describing cells at the inflow boundary of an arbitrary edge $i \in \mathcal{E}$ yields the following off-diagonal elements

\begin{eqnarray}
\sum_{k\neq f(i,1)} |M_{f(i,1),k}| & =& \sum_{k\neq f(i,1)} |A'_{k,f(i,1)} + A'_{f(i,1),k} | \\
 & =& |A'_{f(i,2),f(i,1)}| + \sum_{j \in E^+_i} |A'_{f(i,1),f(j,n_{j})}|\\
 & =&  q_i + \sum_{j \in E^+_i}a_{ij}
\end{eqnarray}

Finally we calculate the sum of off-diagonal cells at the outflow boundary,

\begin{eqnarray}
\sum_{k\neq f(i,n_i)} |M_{f(i,n_i),k}| &=& \sum_{k\neq f(i,n_i)} |A'_{f(i,n_i),k} + A'_{k,f(i,n_i)}|\\
&=& |A'_{f(i,n_i),f(i,n_i-1)}| + \sum_{e \in E^-_i} |A'_{f(e,1),f(i,n_i)}|\\
&=& q_i + \sum_{e \in E^-_i} a_{ei}\label{eq_Tend}.
\end{eqnarray} 

Hence, it remains to verify the two inequalities
\begin{align*}
q_i&\ge \sum_{j \in E^+_i}a_{ij}=q_i\frac{\sum_{j\in
                       E_i^+}q_j}{\sum_{k\in E_i^-}q_k}\quad\text{ and }\quad
                 q_i\ge \sum_{e \in E^-_i}a_{ei}=q_i\frac{\sum_{e\in E_i^-}q_e}{\sum_{k\in E_i^-}q_k}=q_i\;.
                \end{align*}
The second is clear, while the first follows from the volume conservation property (5).

\end{proof}

The natural question whether the discretized system (\ref{eq_syspar}) also is asymptotically stable will be investigated in a future work and is not within the scope of this contribution. So far, it can be stated that changes in the flux direction can cause a transport velocity of zero at distinct edges. This leads to a loss of rank of the operator $A(q)$. A detailed analysis of the algebraic equations (\ref{eq_algeb}) will give insight to the question whether this singularity is persistent and allows to decline asymptotic stability.

\section{Generation of a surrogate model}\label{sec_buildROM}
In this section we focus on the generation of a Galerkin projection $V \in \mathbb{R}^{n \times r}$, which forms the desired surrogate model. As shown in \cite{gugercin_interpolation-based_2009,polyuga_structure_2010,gugercin_structure-preserving_2012}, a Lyapunov stable system can be transformed to scaled co-energy coordinates, $e = L^{-1} \varphi$, where $L^T Q L = \idn$, to obtain a representation in which $Q^e = \idn$,

\begin{align}\label{eq_dyn_coen}
\dot e &= A^e(q) e  + B^e(q) u_T(t)\\
y &= C^e e.
\end{align}

One easily checks that $A^e + (A^{e})^T = L^T(Q A + A^T Q)L\leq 0$ itself fulfills the Lyapunov inequality. Hence, every Galerkin projection $V$ applied to (\ref{eq_dyn_coen}) produces a globally stable ROM ($V^T(A^e + (A^e)^T)V \leq 0$), without the need to imply $Q$ in the reduction process. As shown in section \ref{sec:stab}, $Q$ is diagonal and thus $L = \text{diag}(\sqrt{q_1}^{-1},...,\sqrt{q_n}^{-1}) $. The reduced order model resulting from an arbitrary Galerkin projection is given by  

\begin{equation}\label{eq_sysparred}
\begin{aligned}
\dot e^r &= V^TA^{e,r}(q)V e^r +V^T B^{e,r}(q) u_T(t),\\
y^r &= C^{e,r}V e^r\\
0 &= g(q,y^r,u_H).
\end{aligned}
\end{equation}

Thus we seek a global subspace in which the relevant dynamics of the nonlinear system (\ref{eq_syspar}) evolve, ensuring $y^r \approx y$, $\forall u_T \in \mathcal{U_T}$.\\

Similar to the ideas discussed for parameterized systems in \cite{baur_interpolatory_2011,benner_survey_2015}, we focus on the reduction of the linear time varying problem (LTV) (\ref{eq_dyn_coen}), which splits in two sub-tasks. The first is the choice of a robust set of representative massflows $\mathcal{S}$ from the high dimensional parameter space at which the reduction is performed. The second task is the connection of their corresponding reductions in the dynamical surrogate model. Concerning the choice of the set $\mathcal{S}$, snapshots from a worst case training signal $u_{tr}$ are generated, where the corresponding signal changes with maximal frequency and explores the entire allowed domain,

\begin{align}
u_{tr} \in \mathcal{U}_T:\; \min_t(u_{tr}(t)) = u_l, \; \max_t(u_{tr}(t)) = u_h, \; m\cdot \omega=\hat \omega.
\end{align}

Subsequently, a greedy strategy is applied in frequency space to determine the linearizations $\mathcal{S}$ at which the reduction is performed. Based on all linearizations relevant for the error of possible dynamical simulations $\mathcal{D}$, the linearization with the largest error is added to $\mathcal{S}\subset \mathcal{D}$ iteratively. As an error measure for the approximation quality of the ROM to a linear model, we use a weighted $\hil 2$-norm over the relevant frequency range $W \subset \mathbb{R^+} $ \cite{breiten_near-optimal_2015},

\begin{align}
\nrm{H}^2_{\hil 2(W)} = \frac{1}{2\pi}\integ{W}{}{\nrm{H(i\omega)}_F^2}{\omega}.
\end{align}

The interval of admissible frequencies $W = [W_l,W_h$] is defined by the maximum frequency of the input signal $W_h = \hat f(\omega)$, and a minimal frequency $W_l$ at which the transfer function converges.\\ 

Concerning the second task of connecting local reductions, one approach is to transform dynamically between their subspaces when performing the simulations. Ensuring that the local transformations are efficient and preserve the stability of the reduced model is a nontrivial task since the corresponding parameters are time dependent. The approach used in the following, is to derive local Galerkin projections $V_i$ and combine them to a global projection matrix $V$. Hence, the resulting surrogate model evolves in a global subspace, ensuring stability. Every local projection corresponds to the linear system defined by the volume flow $q_i \in \mathcal{D}$ fixed in time. Combining local projections to a global one includes the risk of creating a prohibitively large reduced system. Here, Krylov based reduction methods proved to be frugal with regard to the resulting reduced order. Consequently, we use a modified version of IRKA \cite{gugercin_mathcalh_2_2008} for the generation of the local reduction. In its original form, IRKA fulfills necessary conditions for a $\hil 2$ optimal interpolation of the full transfer function of the system. Based on initial interpolation points $\sigma$ and directions $b,c$ in the frequency space, it uses a fixed point iteration to ensure that the reduced model interpolates the transfer function $H$ at the mirror images of the reduced poles. Since we address a weighted error norm, convergence of the interpolation points does not fulfill necessary optimality conditions with respect to the modified norm $\nrm{H}^2_{\hil 2(W)}$ \cite{breiten_near-optimal_2015}. For this reason, we only perform $n \in \mathbb{N}$ iterations of IRKA and focus on the variation of the initial interpolation points $\sigma$. This leads to a local transfer function with error $\delta_{q^i} < \bar \delta$, where

\begin{align}
\delta_q = \frac{\nrm{H(q)  - H_r(q)}_{\hil 2(W)}}{\nrm{H(q)}_{\hil 2 (W)}}.
\end{align}

The local reduction strategy is summarized in algorithm \ref{alg:modIRK}.

\begin{algorithm}
\caption{Weighted IRKA}
\label{alg:modIRK}
\begin{algorithmic}
\REQUIRE error bound $\bar \delta$, system matrices $A,B,C$, initial interpolation frequencies $\sigma = \{\sigma_1,...,\sigma_r\}$, iteration limit $N$
\ENSURE local Galerkin projection $V$
\WHILE{$\delta_q \geq \bar \delta$}
\STATE initialize $b_i$ by most dominant singular vector of $H(\sigma_i)$ $\forall i \in [1..r]$
\FOR{i=1 to $N$}
\STATE{$R(\sigma,b) = [(\sigma_1 \idn -A)^{-1}Bb_1,...,(\sigma_{r} \idn -A)^{-1}Bb_{r}]$}
\STATE{$V = \text{orth}\{R(\sigma,b)\}$}
\STATE $\tilde A \leftarrow V^T A V$, $\tilde  B \leftarrow V^T B $, $\tilde  C \leftarrow C V$
\STATE{determine error of current ROM, $\delta(i)$}
\STATE{determine eigenvalues and left eigenvectors, $y_i^* \tilde A = \lambda_i y_i^*$}
\STATE{update $\sigma_i \leftarrow -\lambda_i$, and  $b_i^T \leftarrow 	y_i^* \tilde B$}
\ENDFOR
\STATE{$\delta_q = \min_{i \in [1..N]}{\delta(i)}$}
\STATE{increase number of interpolation points $\sigma$ in relevant range }
\ENDWHILE
\RETURN{V}
\end{algorithmic}
\end{algorithm}

Based on the generation of a local reduction, a global projection matrix is determined by a greedy based selection of parameter vectors. It is started with an initial reduction $V^j$,  with generating space $R^j$. Iteratively, the parameter vector $q^m$ exhibiting the largest deviation to the full transfer function is determined,

\begin{align}
q^m = \argmax_{q \in \mathcal{D}}{\delta_q}. 
\end{align}

Subsequently, the space $R^m$ forming its Galerkin projection $V^m$ is added to the current interpolation space $\mathcal{R}^{in}$, and the current global projection $V$ is updated using a singular value decomposition

\begin{align}\label{eq_sinval}
V &= V(\mathcal{R}^{in};s) = \text{SVD}(R^1,...,R^r;s),
\end{align}

where $s$ is the order of the logarithmic singular value decay, and

\begin{align}
\mathcal{R}^{in} = \{R^i\}_{i = 1,..,r}, \quad q^i \in \mathcal{S},
\end{align}

describes the unification of the local interpolation spaces defined by

\begin{align}
R^i = R^i(\sigma^i,b^i,q^i) = [(\sigma^i_1 \idn -A(q^i))^{-1}B(q^i)b^i_1,...,(\sigma^i_{n_i} \idn -A(q^i))^{-1}B(q^i)b^i_{n_i}].
\end{align}

Since each of the local Galerkin projections $V^i = \text{orth}(R^i)$ supplied by algorithm \ref{alg:modIRK} is orthonormal by construction,  the generating spaces $R^i$ themselves are used in the generation of a global projection $V$. This avoids the problem that each space is given the same weight before entering the singular value decomposition. The procedure is repeated until the maximum local error in the transfer function with respect to all parameter values in the test set $\mathcal{D}$,

\begin{align}
\Delta^\delta = \Delta^\delta(\mathcal{D}, V) = \max_{q \in \mathcal{D}} \delta_q,
\end{align}

is smaller than the global threshold $\bar \Delta$. Using the singular value decomposition (\ref{eq_sinval}) to combine local spaces removes undesired redundancies and decreases the resulting order of the surrogate model. Still, the global projection might include modes which are not redundant but unnecessary for the approximation of the relevant time domain simulations. Therefore, it is useful to vary also the initial reduction starting the greedy selection, which was not added by the largest error criterion and thus might introduce unwanted properties of the final projection. The procedure of calculating a global Galerkin projection is summarized in algorithm \ref{alg:greedfreq}.

\begin{algorithm}
\caption{Determination of global projection matrix $V$: frequency greedy}
\label{alg:greedfreq}
\begin{algorithmic}
\REQUIRE{Evaluation points of volume flow vectors $\mathcal{D}= \{q^1,...,q^{n_\delta}\}$ , error bounds $\bar \delta$, $\bar \Delta$}
\REQUIRE{Local interpolation space $R^i = R^i(\sigma^i,b^i,q^i)$  s.t.  
$\delta_{q^i} < \bar \delta$   $\forall q^i \in \mathcal{D}$}
\ENSURE Global Galerkin projection $V$
\FOR{i=1 : $n_\delta$}
\STATE{Initialize $\mathcal{R}^{in} = R^i$, with $\; q^i \in \mathcal{D}$}
\WHILE{$\Delta^{\delta}\geq \bar \Delta$}
\STATE{determine evaluation with largest deviation $q^m = \argmax_{q \in \mathcal{Q}^\delta}\delta_q$}
\STATE{augment interpolation space $\mathcal{R}^{in} = \mathcal{R}^{in} \bigcup R^m$}
\STATE{determine global Galerkin projection $V = \text{SVD}(\mathcal{R}^{in};s) $}
\STATE{update global error $\Delta^{\delta}$}
\ENDWHILE
\ENDFOR
\RETURN{$V$}
\end{algorithmic}
\end{algorithm}

\begin{figure}
\centering
\includegraphics[scale=0.5]{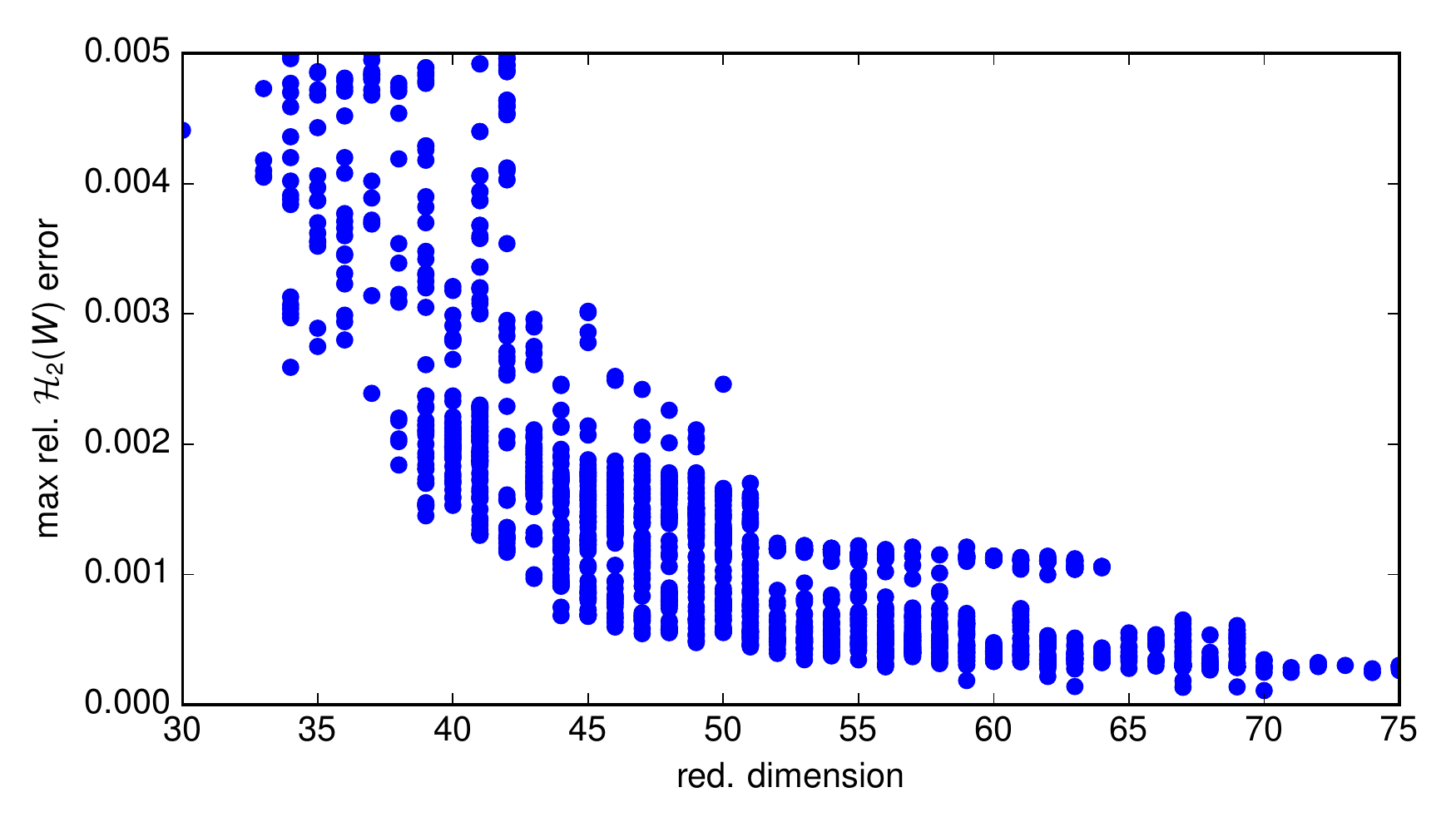}
\caption{Scatter plot of Galerkin projections resulting from algorithm \ref{alg:greedfreq} for network RNS. Visualized is the variation of the initial volume flow field $q^i$ entering the global interpolation space first, as well as the variation $s$ of the singular value decomposition.}
\label{fig_greedy}
\end{figure}

\subsection{Network decomposition}\label{sec_decomp}
To reduce large scale networks more efficiently, a decomposition into several subparts is performed \cite{cheng_reduction_2017}. To this end, the state vector $\varphi$ is decomposed to a part corresponding to a main network $\varphi^0$, and $s \in \mathbb{N}$ subnetworks collected by the state vector $\tilde \varphi$. The main network is defined as the object to which the external input energy density $u_T$ is applied. The entire system dynamics are given by

\begin{eqnarray}\label{eq_decom}
\pmat{\dot \varphi^0 \\ \dot{\tilde \varphi}} &=& \pmat{A^0(q)& \\  & \tilde A(q)} \pmat{\varphi^0\\ \tilde \varphi}  + \pmat{B^0(q) & \\  & \tilde B(q)} \pmat{u_T\\ \tilde y}\label{eq_sysnet_A} \\
\pmat{\tilde y\\y} &=&\pmat{C^0_{\tilde y} &  \\C^0_h & \tilde C} \pmat{\varphi^0\\ \tilde\varphi }\label{eq_sysnet_C} \\
0 &=& g(u_H,y,q). \label{eq_sysnet_algeb}
\end{eqnarray}

Here $\tilde y$ is the input to the subnetworks measured by energy densities of the main network. Subsequently, main- and subnetworks are reduced separately employing arbitrary, local Galerkin projections $V_i, \; i \in [0,..,s]$. Note that by eliminating $\tilde y$ in (\ref{eq_decom}), the block diagonal system matrix $A$ obtains a contribution $\tilde B C^0_{\tilde y}$ in the lower left block, coupling the diagonal blocks. The resulting system matrix is identical to the formulation (\ref{eq_syspar}), for which section \ref{sec:stab} proved Lyapunov stability. Hence, this property is preserved when reducing different networks separately. A key saving of computational complexity introduced by the decomposition arises in the assembly of the parametric operator $A(q)$. When reducing the entire network by projection $V$, the formation $A^r(q)$ requires $n_q$ multiplications of the scalar $q_i$ with the matrix $A^r_i$. In contrast, the reduced system matrix of the decomposed network preserves the block structure, and only the local volumeflows in the subnetwork are applied to the system matrix $A^i\; \in [0,..,s]$ of the subnetwork.

\section{Numerical analysis of the surrogate model}\label{sec_numerics}
In this section, we study the approximation quality and the runtime of the reduced model presented in section \ref{sec_buildROM} for different input scenarios and different real world networks. In the benchmark, both the time integration as well as higher order hyperbolic schemes are taken into account. Furthermore, different approximation qualities defined by the number of finite volume cells are studied. Computations presented in this section are performed employing MATLAB(R) R2016b on an Intel(R) XEON(R) CPU E5-2670 processor @ 2.60GHz. The number of finite volume cells is distributed to pipelines according to

\begin{align}\label{eq_cells}
n_i = \max\lbrace{n_\text{min},\text{round} \left( c_r \frac{L_i}{L_r}\frac{v_r}{v_i} \right) \rbrace} \; \forall i \in \mathcal{E},
\end{align}

where $L_r,v_r,c_r$ denote length, velocity and the expected number of cells of the reference pipeline. Eq. (\ref{eq_cells}) aims at minimizing the variance of the CFL rations on each pipeline. The reference pipeline is thus the one with the largest CFL condition, where its velocity is obtained from a simulation employing typical inputs. The considered error in time domain is defined in (\ref{eq_err_time}).\\

The first example represents a street with 32 consumers and one loop forming 33 independent volume flows (parameters), \fig\ref{fig_net}. We fix the consumption of the houses to typical average values constant in time. For the training phase the input signal $u_{in} \in \mathcal{U_T}$ is used,

\begin{align}\label{eq:siginsamp}
u_{in}(t) = 0.4 + 0.2 \cos(wt), \quad \frac{2\pi}{\omega} = 14 \times 10^3s
\end{align}

spanning the allowed range for the energy density. This input signal is referred to as the in-sample signal. The Galerkin reduction resulting from algorithm \ref{alg:greedfreq} is further tested for the (out-of-sample) signal $u_{out} \in \mathcal{U_T}$

\begin{align}\label{eq:sigoutsamp}
u_{out}(t) =c_0 + \sum_{i=1}^2 [ c_i\cos(i \omega t) + s_i\sin(i \omega t)], \quad \frac{2\pi}{\omega} = 28 \times 10^3s, 
\end{align}

with the coefficients $(c_0,c_1,c_2)=(0.37,-0.078,0.089), \; (s_1,s_2) = (c_1,-c_1) $.\\

\begin{figure}
\begin{minipage}[t]{0.49\textwidth}
\includegraphics[width=\textwidth]{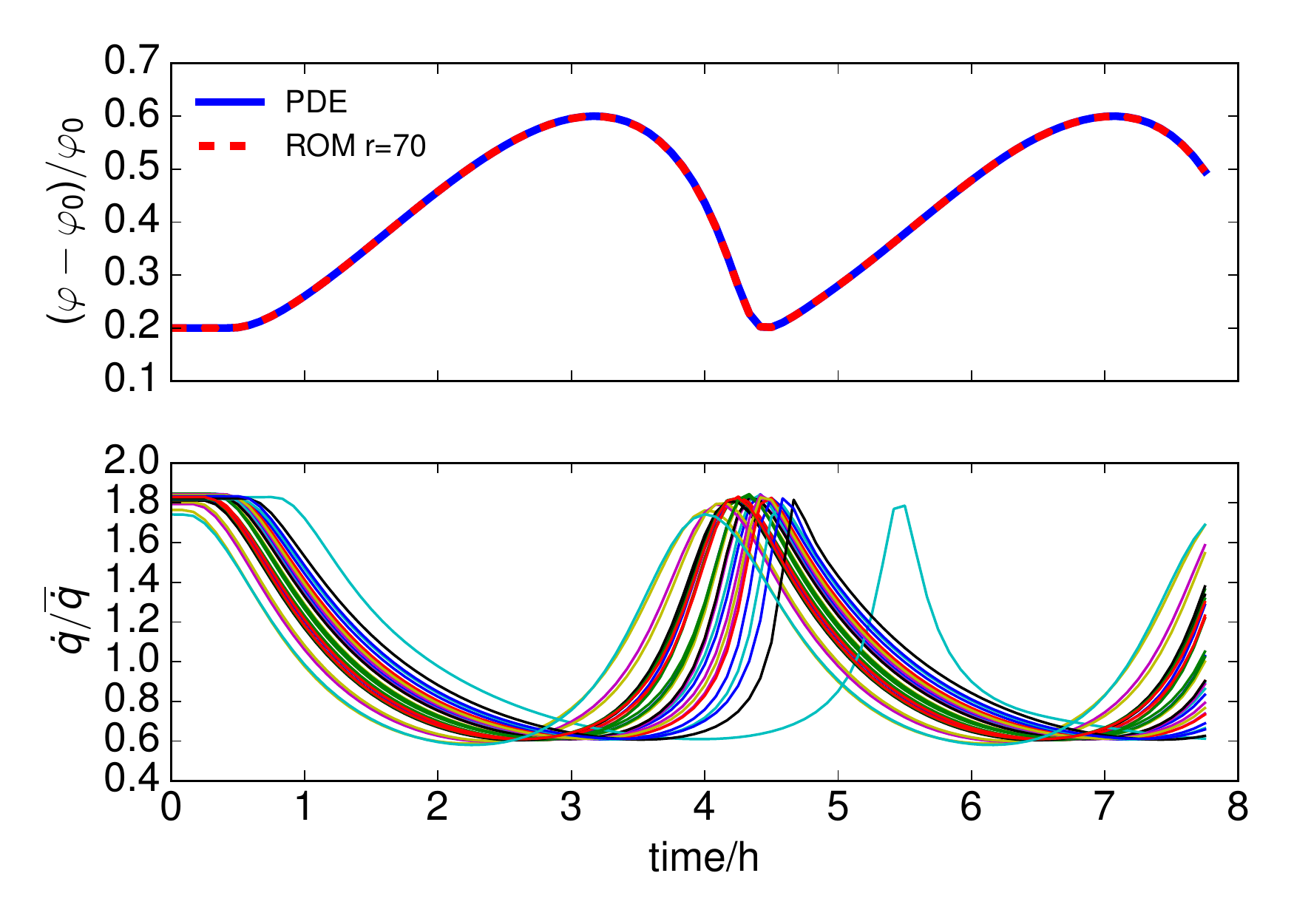}
\begin{picture}(0,0)
\put(10,170){(a)}
\end{picture}
\end{minipage}
\begin{minipage}[t]{0.49\textwidth}
\includegraphics[width=\textwidth]{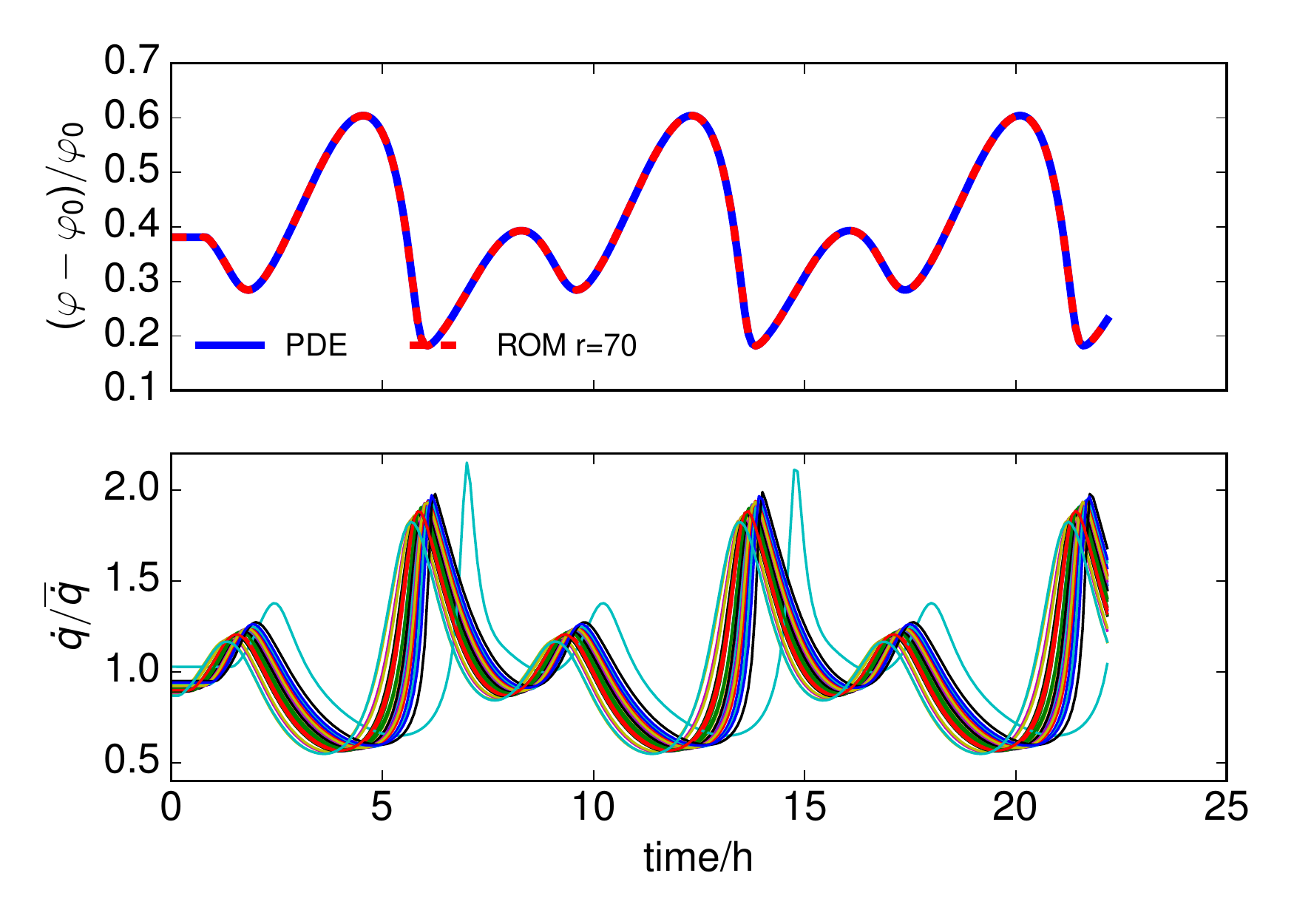}
\begin{picture}(0,0)
\put(-10,170){(b)}
\end{picture}
\end{minipage}
\caption{Comparison of the PDE solution and the reduced scheme with $r=70$ states for the street network RNS for the in-sample signal(a), and the out-of-sample signal(b). The lower part of both plots shows the normalized temporal variation of volume flows.}
\label{fig_street_sign}
\end{figure}
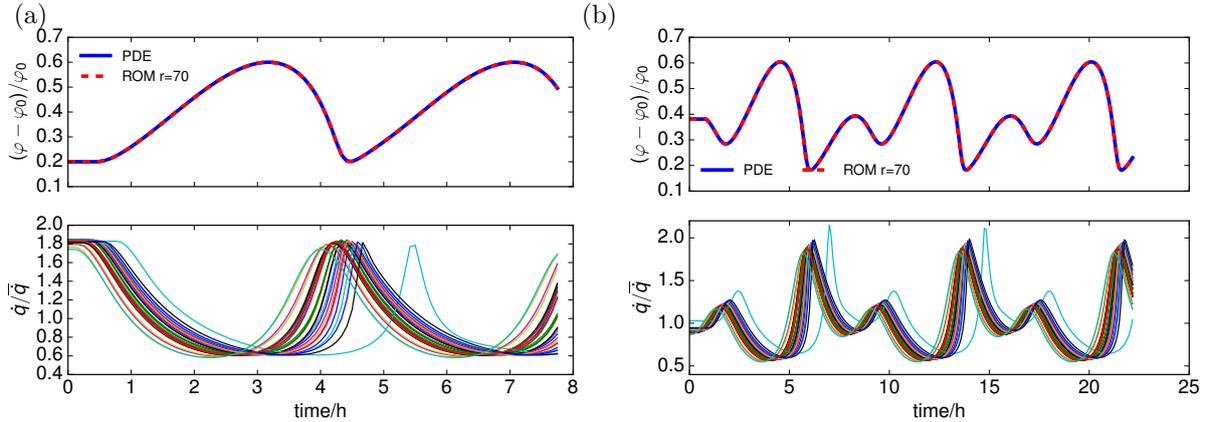

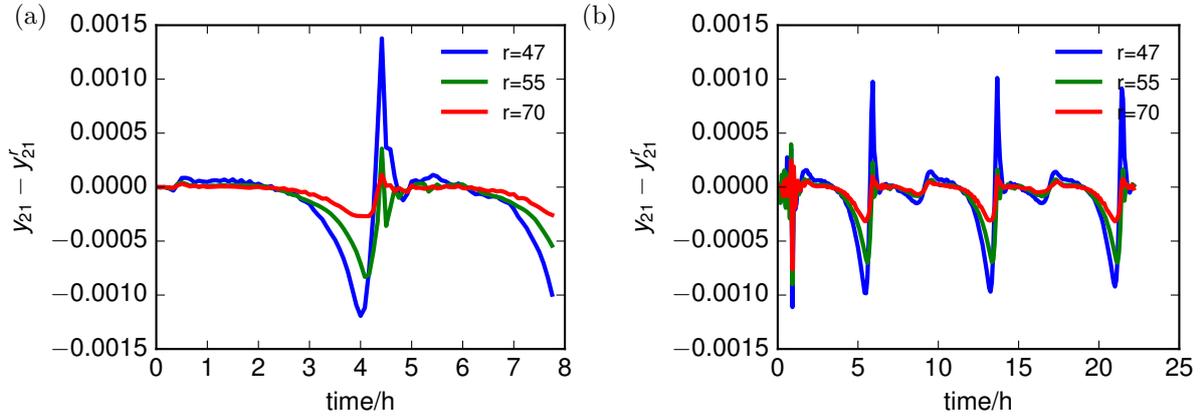
\begin{figure}
\begin{minipage}[t]{0.49\textwidth}
\includegraphics[width=\textwidth]{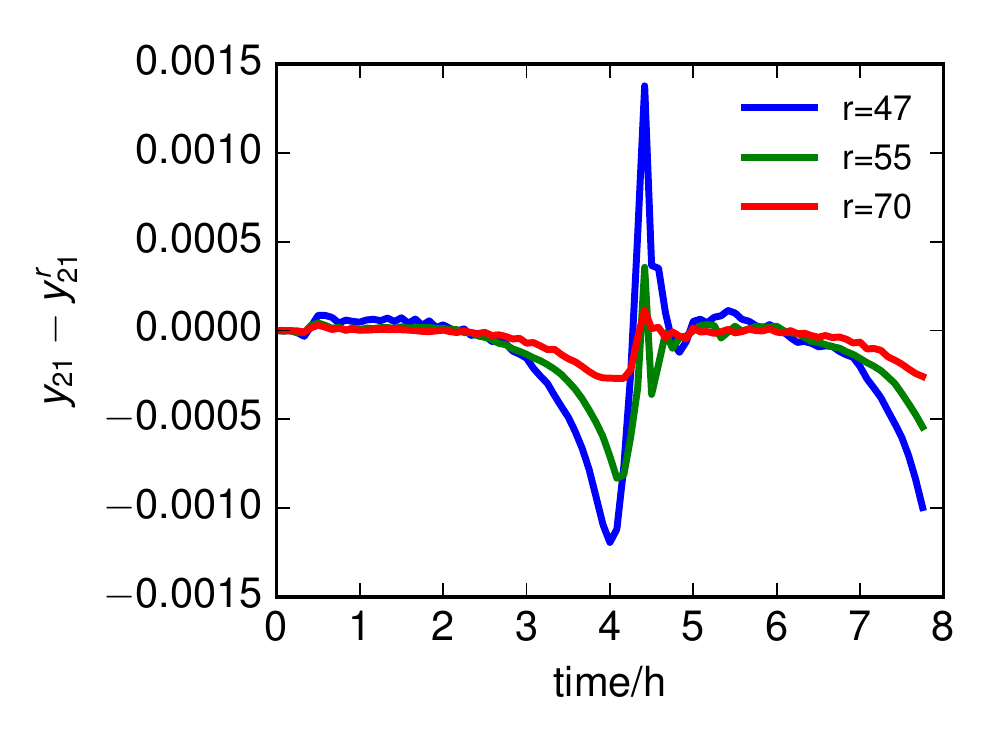}
\begin{picture}(0,0)
\put(10,170){(a)}
\end{picture}
\end{minipage}
\begin{minipage}[t]{0.49\textwidth}
\includegraphics[width=\textwidth]{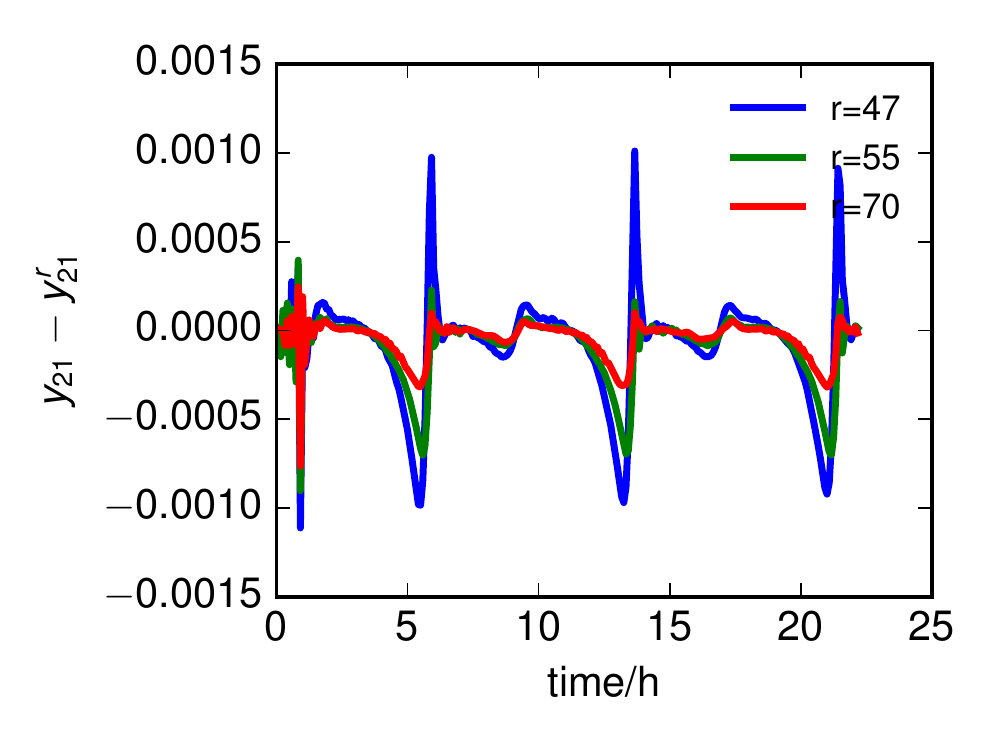}
\begin{picture}(0,0)
\put(-10,170){(b)}
\end{picture}
\end{minipage}
\caption{Absolute error of output 21 in network RNS showing the largest distance from source to output. The size of the ROM is altered comparing the in-sample input signal(a) to the out-of-sample signal(b).}
\label{fig_street_err}
\end{figure}

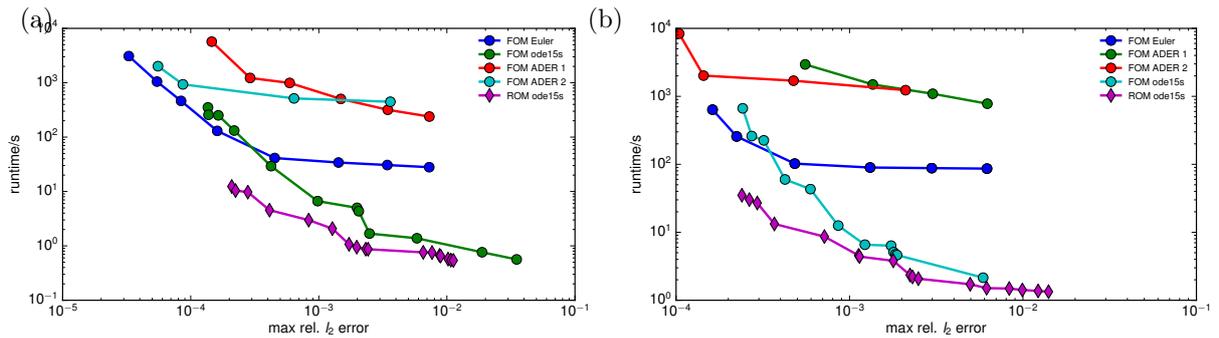
\begin{figure}
\begin{minipage}[t]{0.49\textwidth}
\includegraphics[width=\textwidth]{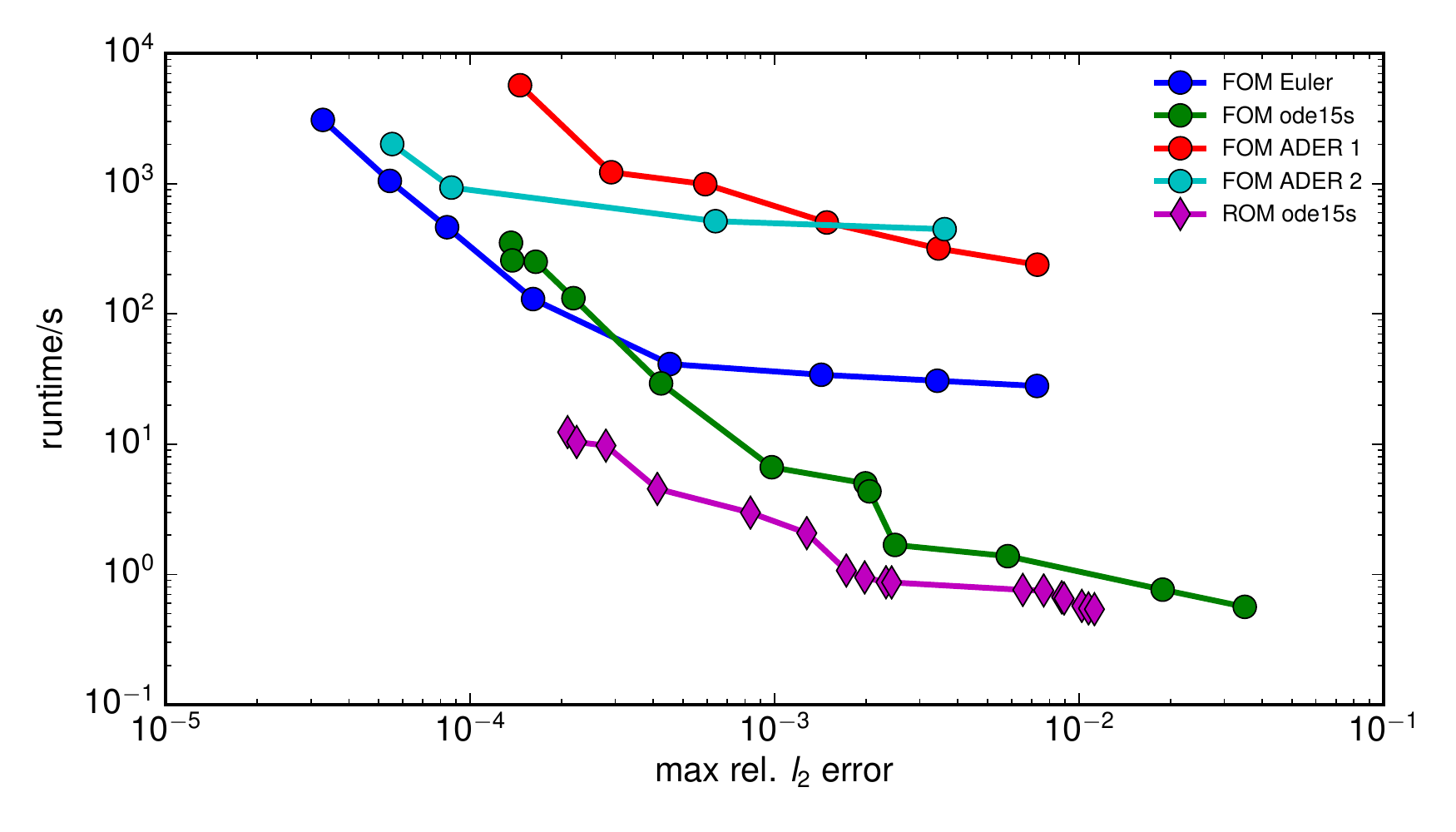}
\begin{picture}(0,0)
\put(10,135){(a)}
\end{picture}
\end{minipage}
\begin{minipage}[t]{0.49\textwidth}
\includegraphics[width=\textwidth]{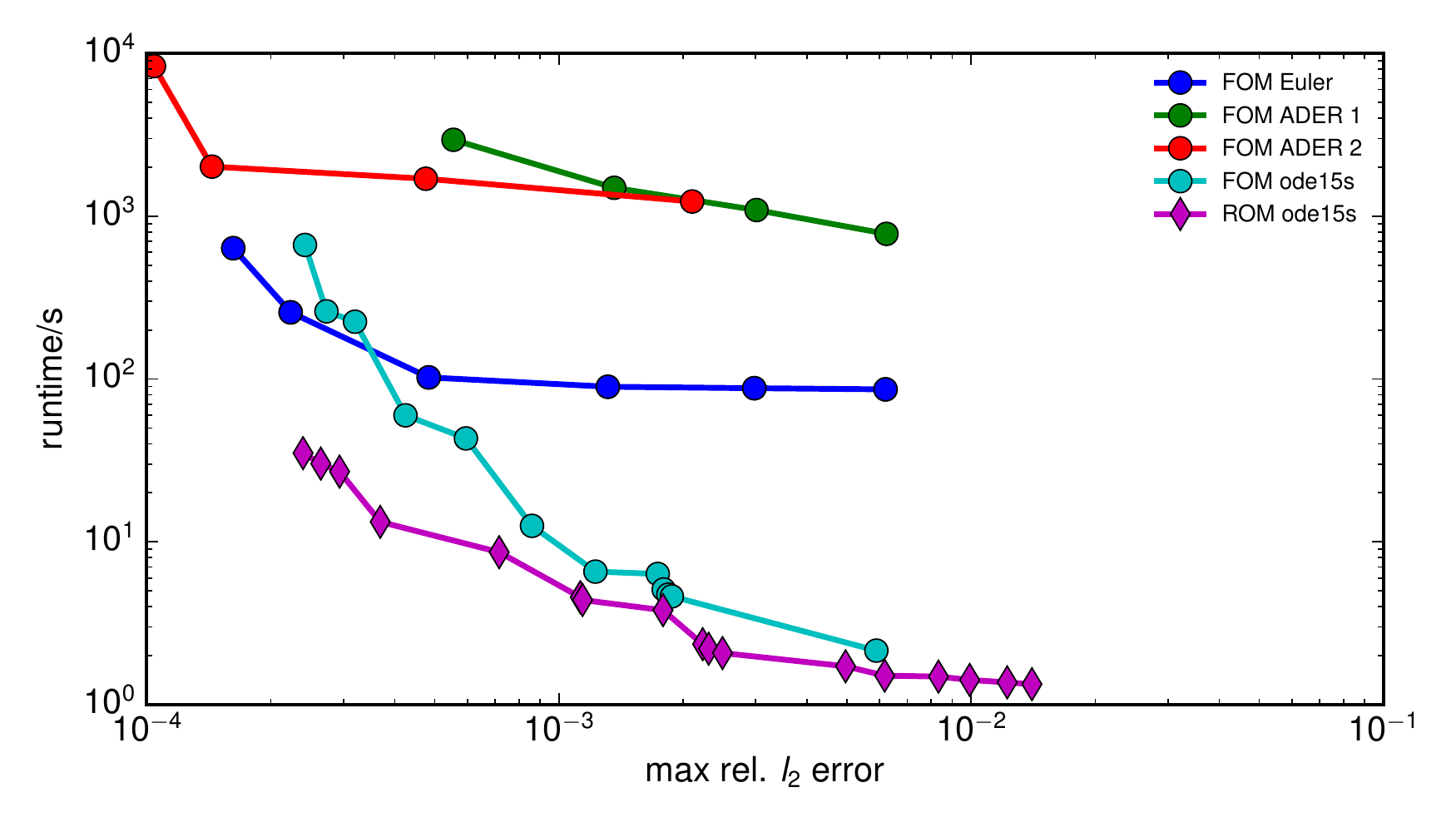}
\begin{picture}(0,0)
\put(-10,135){(b)}
\end{picture}
\end{minipage}
\caption{Runtime versus time domain error $\Delta^T$ for the reference network RNS varying both time- and space discretizations. ode15s denotes the MATLAB(R) implicit ODE solver.}
\label{fig_street_integ}
\end{figure}

\fig\ref{fig_street_sign} a), b) show the PDE solution for the two input signals (\ref{eq:siginsamp}, \ref{eq:sigoutsamp}) and the output $y$ of the ROM exhibiting $r=70$ states. The original cosine wave is inclined asymmetrically due the time dependent transport velocity. The ROM is able to precisely reflect the nonlinear dynamics of the street network for both in-sample and out-of-sample signals. The estimate for the PDE solution is obtained by extrapolation of simulation results with increasingly fine spatial resolution. The reference model used to this end is the Upwind scheme. The initial state for all simulations equals the equilibrium state defined by the initial control, $\varphi(x,t=0)=u(t=0), \; \forall x$. As a consequence, both input signals $u_{in}, u_{out}$ are continuous in time, while $u_{in}$ is in addition differentiable, since $\dot u_{in} = 0$.\\ 

Employing an upwind discretization of the PDE, the number of finite volume cells is increased to compare the development of runtime, approximation error and size of full and reduced models. To further investigate the benefits of the ROM for different resolutions of the spatial discretization, the number finite volume cells is varied and different time integration schemes are applied to both full and reduced models, cf. \fig\ref{fig_street_integ}. More precisely, FOM Euler (FOM ode15s) denotes an upwind discretization of the full order model, where the time integration is performed using the explicit  Euler scheme (the implicit, adaptive Runge-Kutta-method ode15s supplied by MATLAB(R)). Additionally, FOM ADER 1(2) denote ADER schemes of order 1(2) for hyperbolic conservation laws, which are integrated by an explicit Euler scheme. Finally ROM ode15s is the reduced order model of the upwind discretization in space, integrated in time using ode15s.\\ 

Going from higher to smaller approximation errors, two regimes are visible. In the first one,  the numerical transport velocities reflected by the CFL number , $\lambda_i = n_i  v_i/L_i, \; i \in \mathcal{P}$ synchronize over all pipelines.  Here, explicit time integration schemes whose timesteps are restricted by $\hat \lambda = \max _{i \in \mathcal{P}}\lambda_i$ show a high computational effort compared to implicit time integration schemes for higher errors. This effect aggravates for ADER \cite{borsche_high_2016} schemes of order $o$ since they require a minimum number of $o$ cells, further diminishing the maximum step size. Increasing the number of finite volume cells, $\hat \lambda$ remains constant in the synchronization phase and explicit timestepping models improve in accuracy without larger additional costs. Towards the end of the synchronization, $\hat \lambda$ increases in the number of cells $n$ and the computational costs increase superlinear in $n$. The implicit time integration of the Upwind scheme is in particular fast in the synchronizing regime in which the problem is stiff. With increasing number of cells, the stiffness of the problem reduces, and the benefits of the full order implicit time integration vanish as well. Comparing the implicit time integration for full and reduced order models, the speed-up increases towards finer resolutions. As a consequence, the reduced model shows the smallest runtime within the considered error range. This range is suitable for optimal control purposes, in which an exact solution is not the central objective, but a fast estimate of it.\\

In the second example we consider a larger network with $n_h=333$ consumers, $n_p = 775 $ pipelines and $n_L=6$ loop flows, which represents the network topology of an existing district, cf. \fig\ref{fig_net} b). Due to multiple loops in the left part of the network (highlighted in green), changes in the fluxdirection occur for the chosen control. The latter modifies the coupling structure of the finite volume cells in the reverting pipelines, since the external cells delivering inputs change with the inflow boundary. The corresponding system matrix $A(q)$ changes its sparsity structure, in contrast to the case of constant flux direction, where a change of volume flow modifies the relative weights of the entries. When generating projections for this network incorporating these nonlinearities, a prohibitively large reduced dimension results. Hence, the network is decomposed into subnetworks, cf. section \ref{sec_decomp}. To this end, the subnetwork exhibiting all flux changes is separated and as the only one not reduced. For the remaining subparts, algorithm \ref{alg:greedfreq} is applied. Concerning the dynamics, changes in the flux direction introduce a highly nonlinear response of the system, compare \fig\ref{fig_neub_integ} (a). Still, the output of the ROM in all subnetworks is accurate enough to recover the complex dynamics in the flux changing subnetwork. Focusing on the runtime comparison for different accuracies, the same observations as for the street network hold true. Furthermore the speed-up of the reduced order model is shifted towards smaller relative errors.\\

\begin{figure}
\begin{minipage}[t]{0.49\textwidth}
\includegraphics[width=\textwidth]{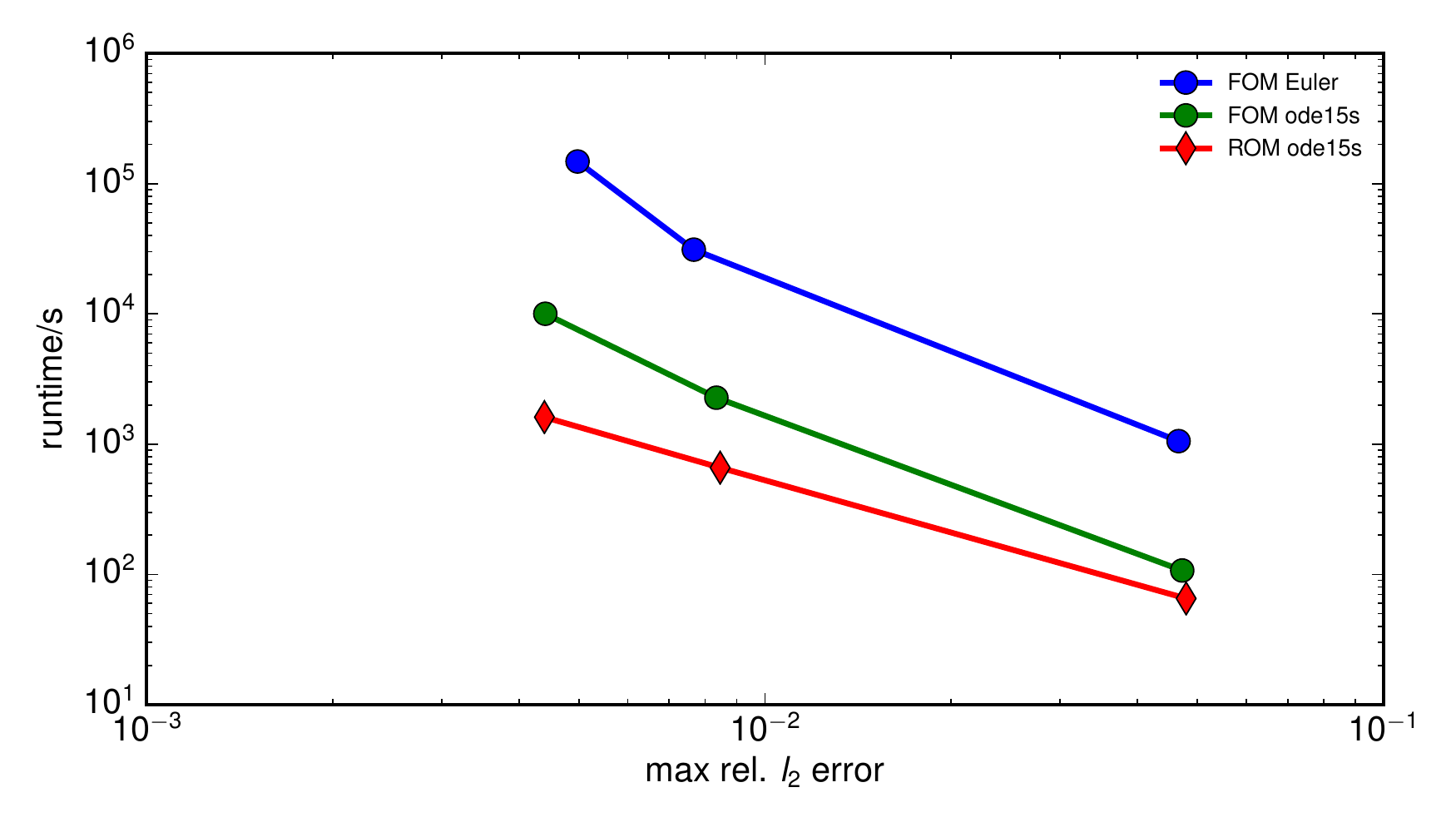}
\begin{picture}(0,0)
\put(10,140){(a)}
\end{picture}
\end{minipage}
\begin{minipage}[t]{0.49\textwidth}
\includegraphics[width=\textwidth]{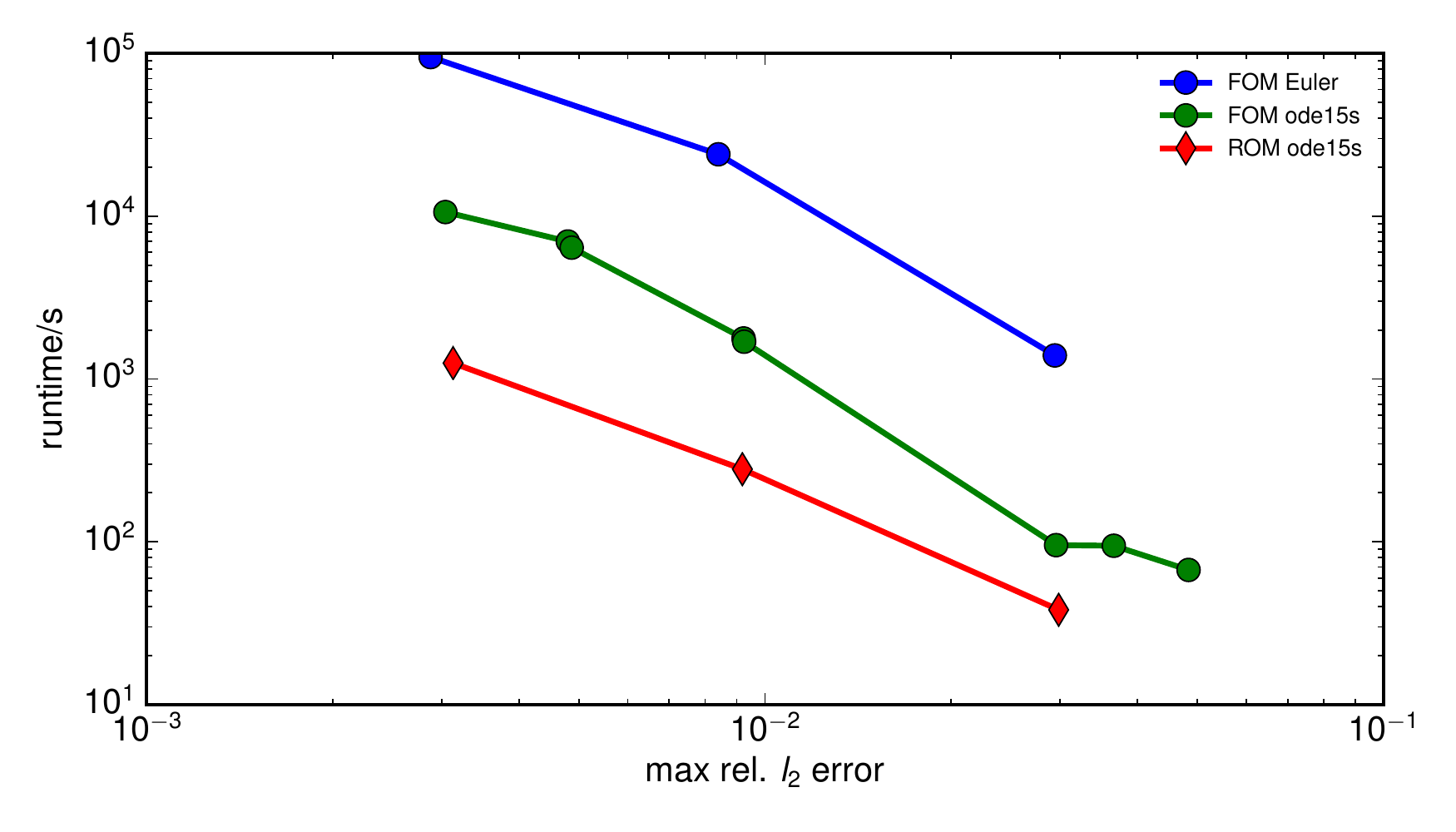}
\begin{picture}(0,0)
\put(-10,140){(b)}
\end{picture}
\end{minipage}
\caption{Runtime versus time domain error $\Delta^T$ for the reference network RND varying both time- and space discretizations. ode15s denotes the MATLAB(R) implicit ODE solver.}
\label{fig_neub_integ}
\end{figure}
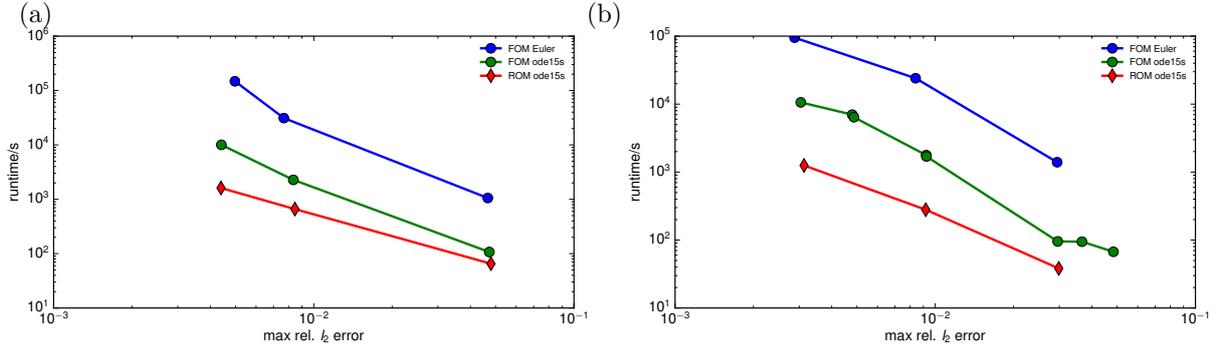

\begin{figure}
\begin{minipage}[t]{0.49\textwidth}
\includegraphics[width=\textwidth]{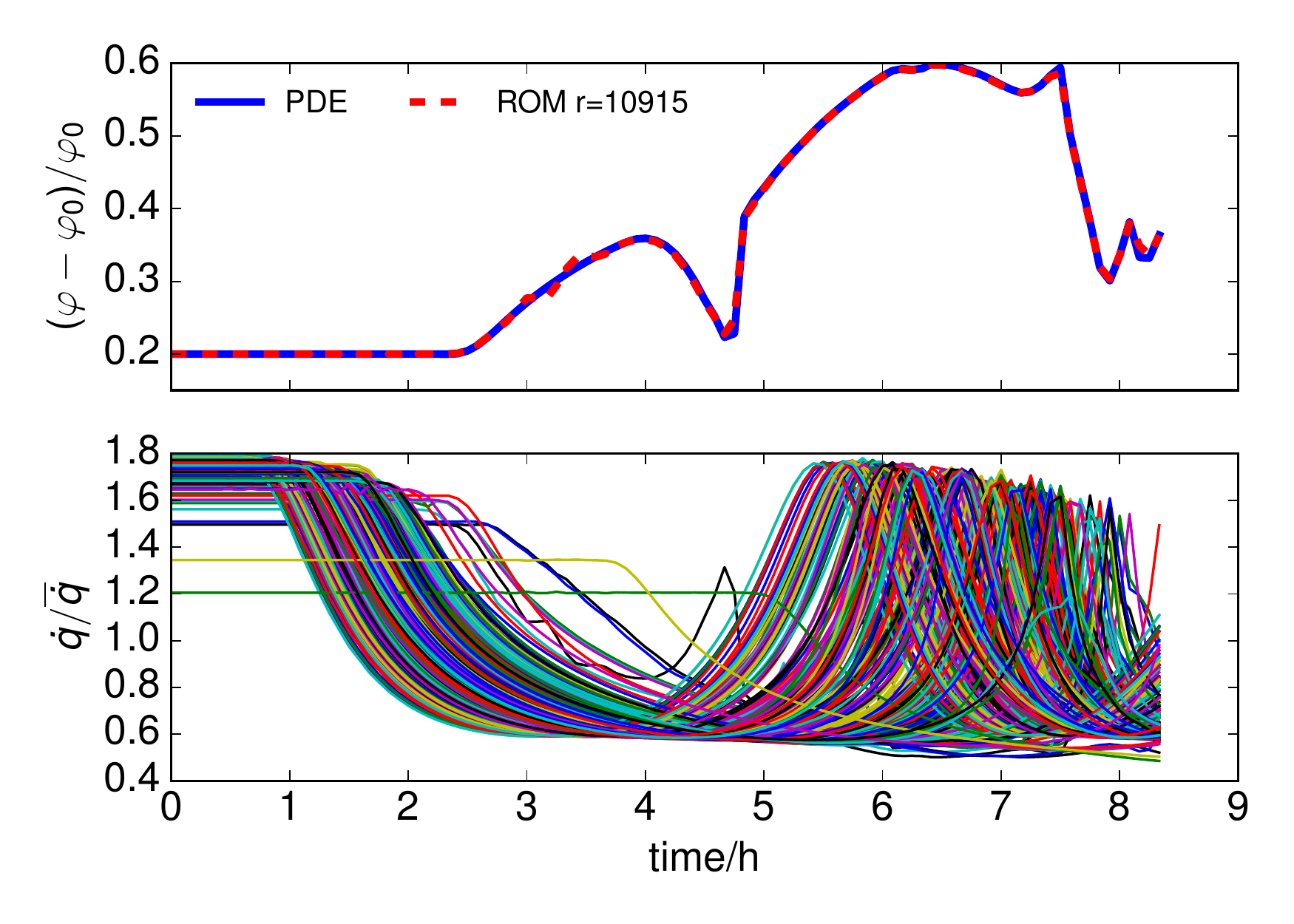}
\begin{picture}(0,0)
\put(10,170){(a)}
\end{picture}
\end{minipage}
\begin{minipage}[t]{0.49\textwidth}
\includegraphics[width=\textwidth]{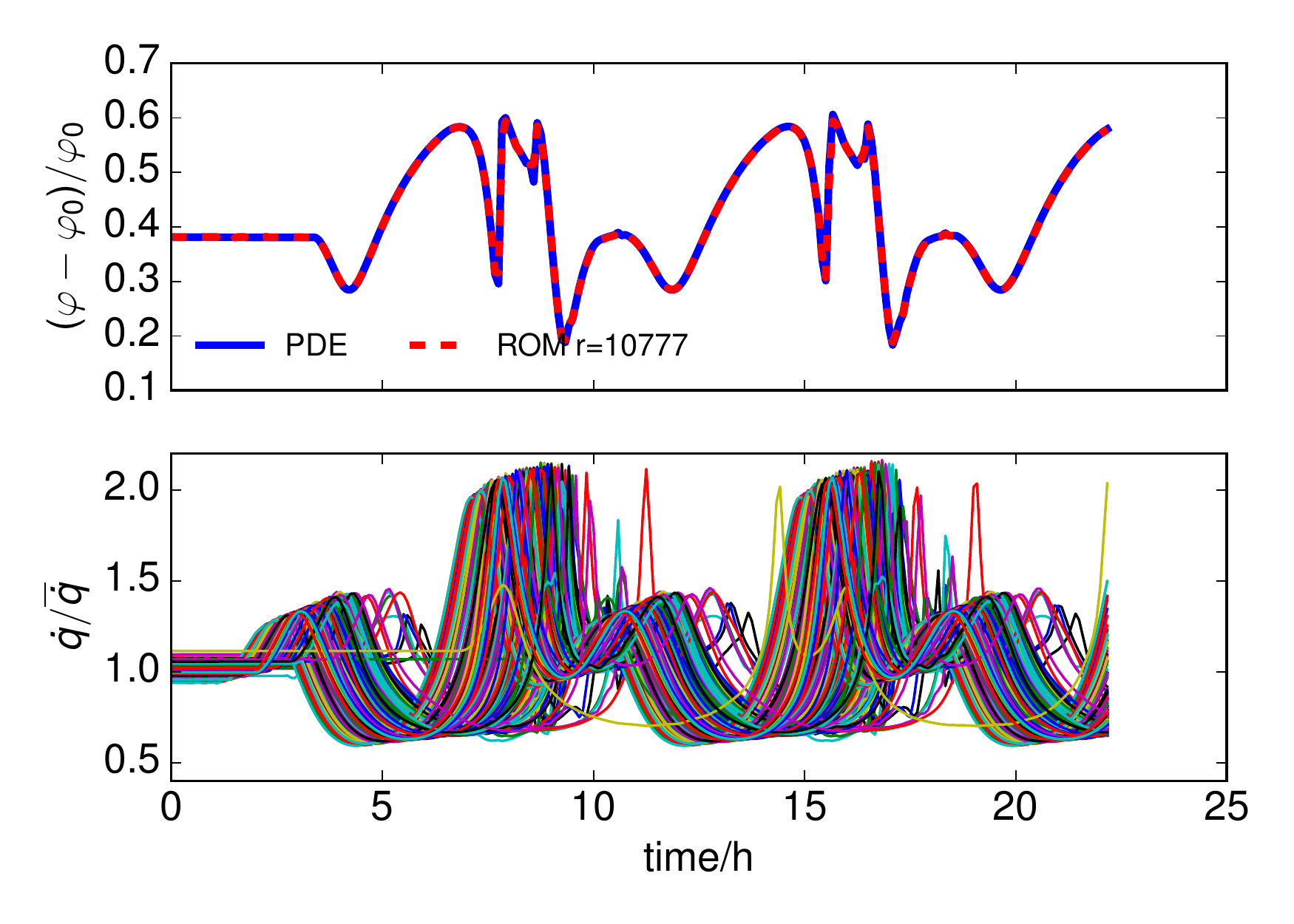}
\begin{picture}(0,0)
\put(-10,170){(b)}
\end{picture}
\end{minipage}
\caption{Comparison of the PDE solution and the reduced scheme for the district network RND for the in-sample input(a), and the out-of-sample signal(b). The lower part of both plots shows the temporal variation of volume flows.}
\label{fig_neub_expsim}
\end{figure}
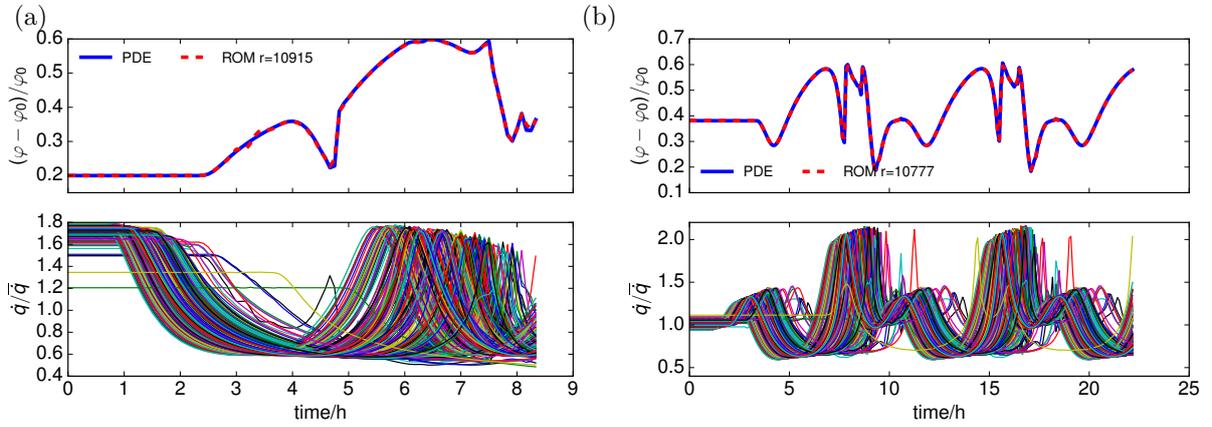

\begin{figure}
\begin{minipage}[t]{0.49\textwidth}
\includegraphics[width=\textwidth]{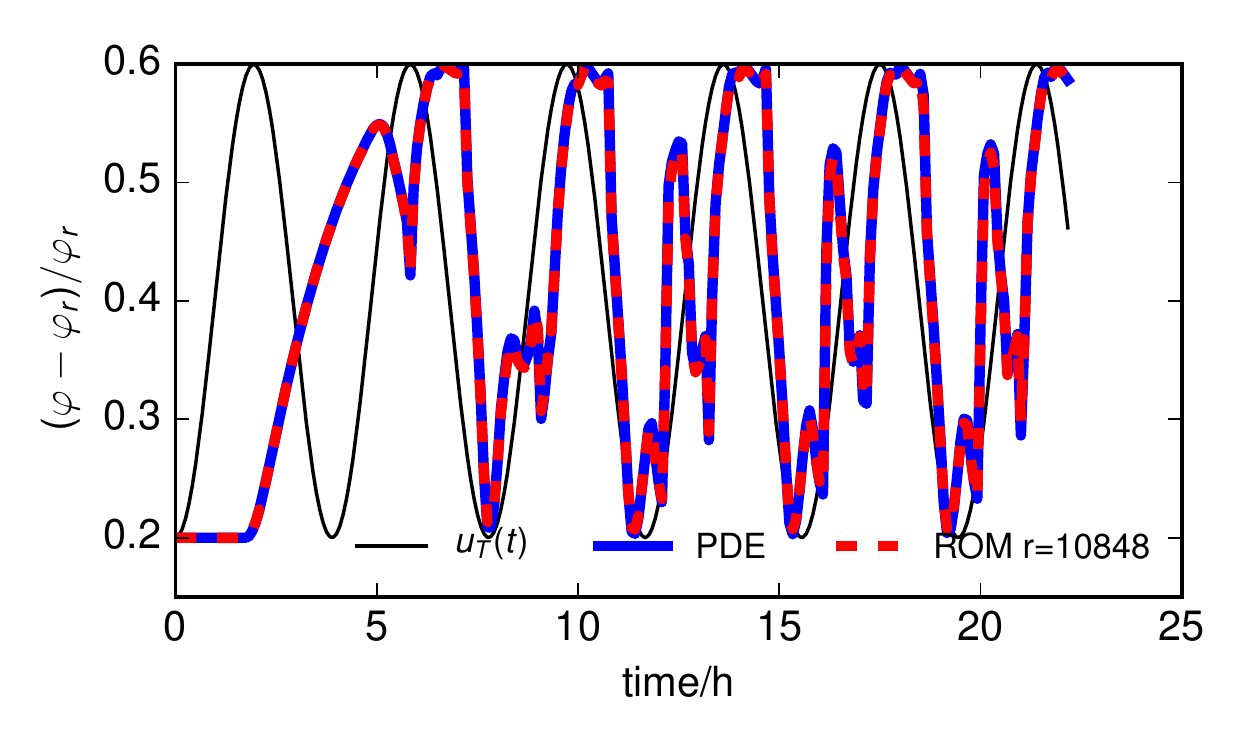}
\begin{picture}(0,0)
\put(10,170){(a)}
\end{picture}
\end{minipage}
\begin{minipage}[t]{0.49\textwidth}
\includegraphics[width=\textwidth]{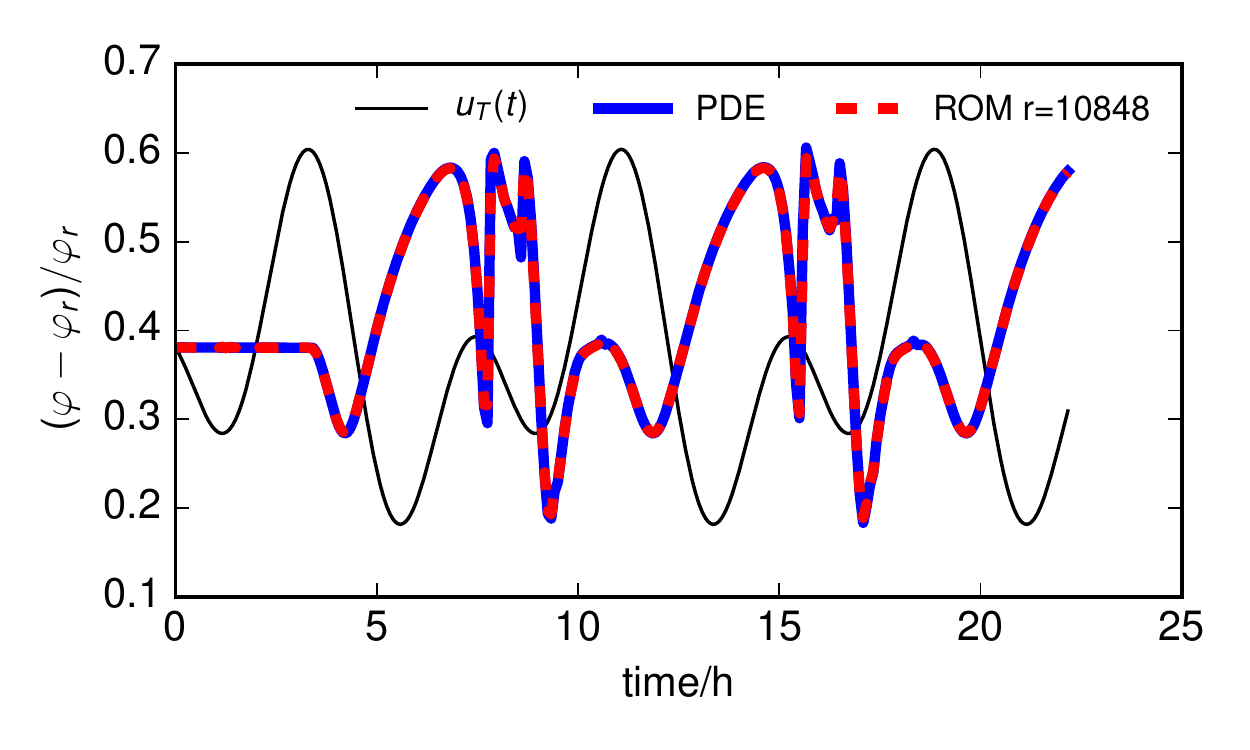}
\begin{picture}(0,0)
\put(-10,170){(b)}
\end{picture}
\end{minipage}
\caption{Relative energy density versus runtime in the network RND. The displayed signals are the input, the PDE solution, and the reduced scheme for the in-sample signal(a), and the out-of-sample signal(b).}
\label{fig_neub_input}
\end{figure}
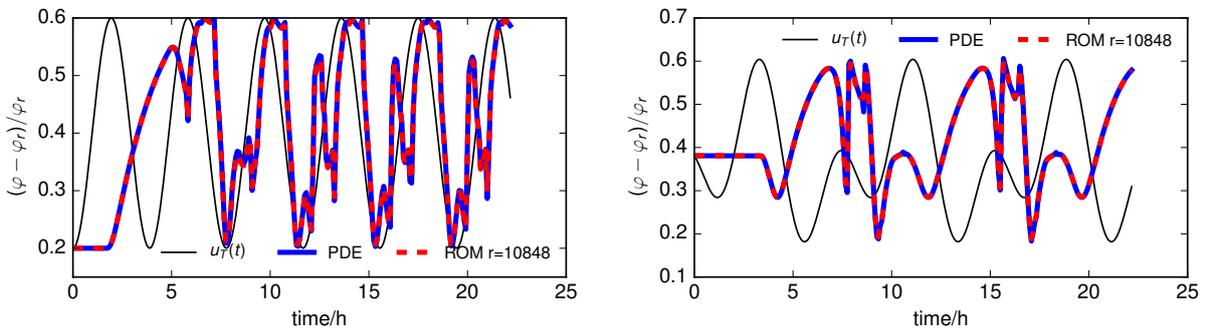

\begin{figure}
\begin{minipage}[t]{0.4\textwidth}
\includegraphics[width=\textwidth]{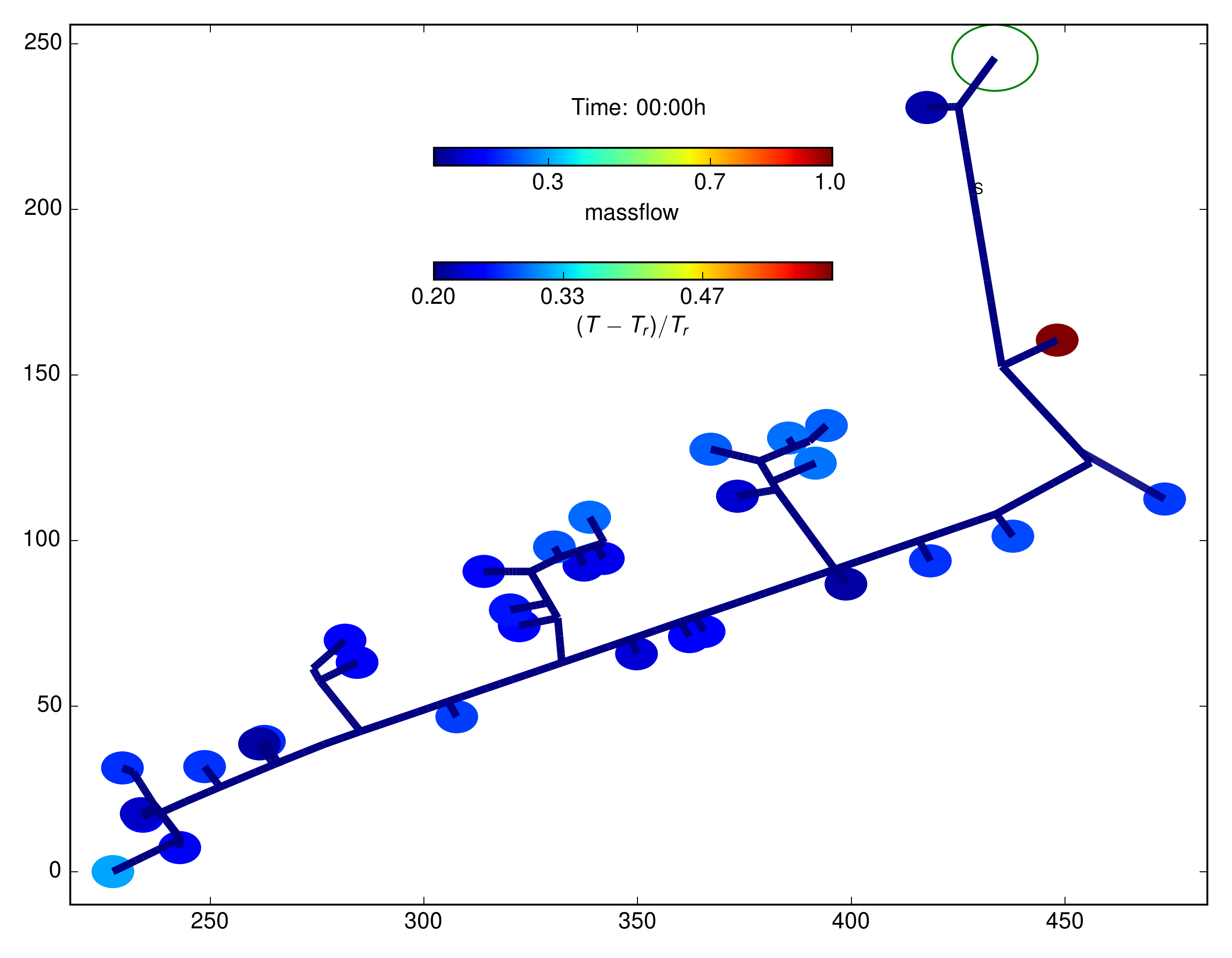}
\begin{picture}(0,0)
\put(20,130){(a)}
\end{picture}
\end{minipage}
\begin{minipage}[t]{0.6\textwidth}
\includegraphics[width=\textwidth]{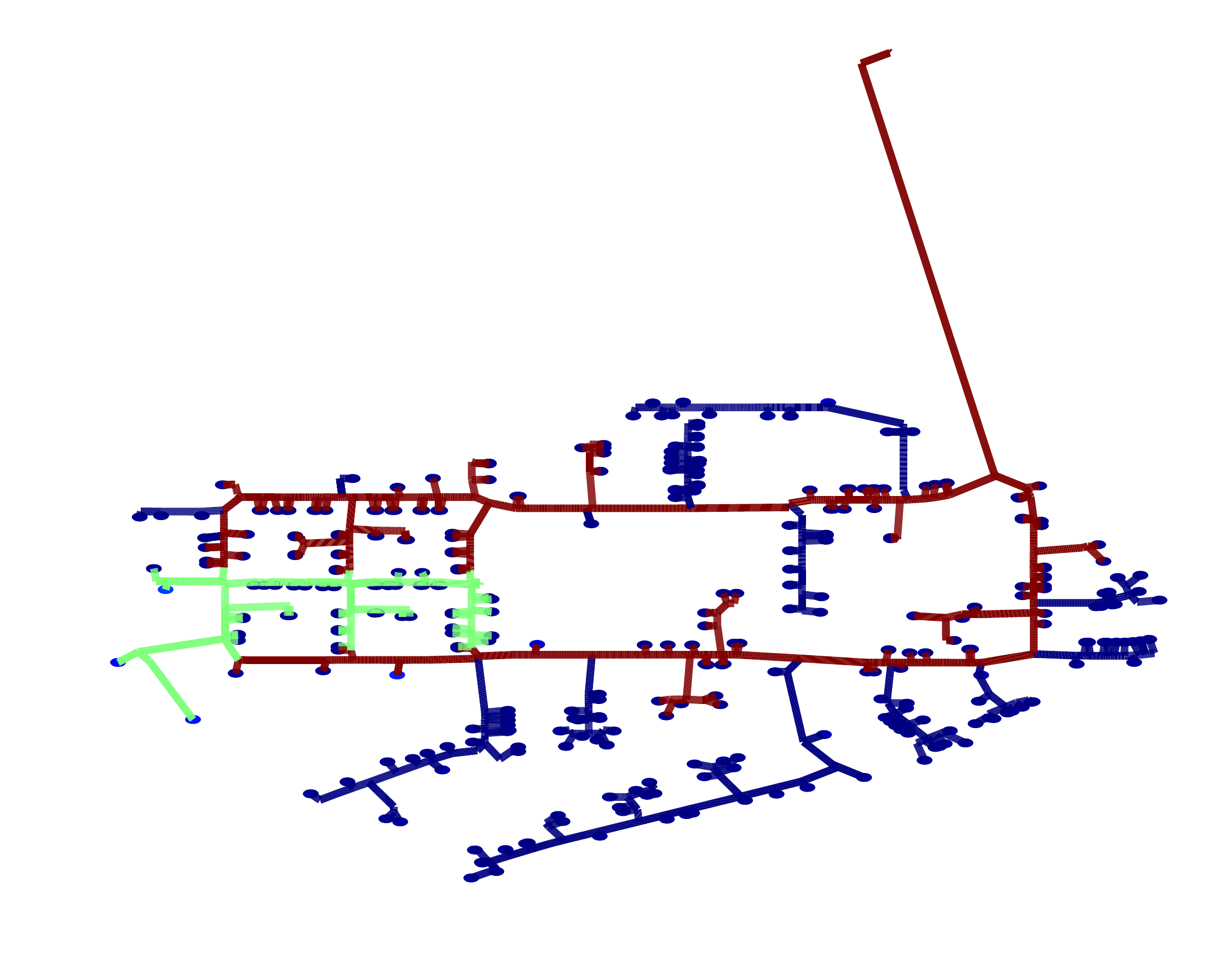}
\begin{picture}(0,0)
\put(20,130){(b)}
\end{picture}
\end{minipage}
\caption{(a) Reference network "RNS" representing a street with 32 consumers, 81 pipelines and 1 loop, modeling a typical street. (b) Reference network "RND" with 333 consumers, 775 pipelines and 6 loops modeling an existing network for a city district. Colors show the decomposition of the network to main network (red), subnetworks (blue) and the subnetwork exhibiting changes of the flux direction (green).}
\label{fig_net}
\end{figure}

\section{Conclusions}
\label{sec:conclusions}
In this paper we presented a route to efficiently generate stable and fast surrogate models for the simulation of advection dominated hyperbolic DAEs at the example of district heating networks. By conserving the algebraic part of the DAE, Lyapunov stability can be shown and the construction of the corresponding global energy matrix is presented. By constructing a Galerkin projection in the offline phase based on frequency space samples, the reduced model is both stable and generically applicable for different admissible inputs. The resulting surrogate model allows for a speed up of up to one order of magnitude compared to full order hyperbolic schemes in the error range of interest. It will be interesting to further check the effectiveness of the model towards the use in optimal control. Furthermore, we will study the modification of size and runtime of the surrogate, when refining the model. Possible refinements are a sink term in the energy transport, acceleration in the conservation of momentum, and a more precise model for frictional processes within the pipelines.

\section*{Acknowledgments}
We acknowledge the financial support by the Federal Ministry of Education and Research of Germany in the framework of the project \textit{"Verbundprojekt im BMBF-Programm "Mathematik f\"ur Innovationen" - EiFer: Energieeffizienz durch intelligente Fernw\"armenetze"} (F\"orderkennzeichen: 05M18AMB).\\

%\bibliography{../../used}

\end{document}